\documentclass[11pt]{amsart}
\usepackage[active]{srcltx}
\usepackage{calc,amssymb,amsthm,amsmath,amscd, ulem}
\usepackage{alltt}
\usepackage[left=1.35in,top=1.25in,right=1.35in,bottom=1.25in]{geometry}
\RequirePackage[dvipsnames,usenames]{color}

\DeclareFontFamily{OMS}{rsfs}{\skewchar\font'60}
\DeclareFontShape{OMS}{rsfs}{m}{n}{<-5>rsfs5 <5-7>rsfs7 <7->rsfs10 }{}
\DeclareSymbolFont{rsfs}{OMS}{rsfs}{m}{n}
\DeclareSymbolFontAlphabet{\scr}{rsfs}

\newtheorem{theorem}{Theorem}[section]
\newtheorem{lemma}[theorem]{Lemma}
\newtheorem{proposition}[theorem]{Proposition}
\newtheorem{corollary}[theorem]{Corollary}

\theoremstyle{definition}
\newtheorem{definition}[theorem]{Definition}
\newtheorem{example}[theorem]{Example}

\theoremstyle{remark}
\newtheorem{remark}[theorem]{Remark}

\newcommand{\sH}{\scr{H}}

\newcommand{\bN}{\mathbb{N}}
\newcommand{\N}{\mathbb {N}}
\newcommand{\bQ}{\mathbb{Q}}
\newcommand{\bR}{\mathbb{R}}

\DeclareMathOperator{\Spec}{{Spec}}
\DeclareMathOperator{\End}{{End}}
\newcommand{\blank}{\underline{\hskip 10pt}}
\newcommand{\sF}{\scr{F}}
\newcommand{\sG}{\scr{G}}
\DeclareMathOperator{\Supp}{{Supp}}
\DeclareMathOperator{\Sing}{{Sing}}
\newcommand{\sM}{\scr{M}}
\newcommand{\sL}{\scr{L}}
\DeclareMathOperator{\Hom}{Hom}
\DeclareMathOperator{\Conv}{Conv}
\DeclareMathOperator{\sHom}{{\sH}om}
\DeclareMathOperator{\id}{{id}}

\DeclareMathOperator{\Div}{{div}}
\DeclareMathOperator{\Frac}{{Frac}}
\newcommand{\tensor}{\otimes}
\newcommand{\bF}{\mathbb{F}}
\newcommand{\bm}{\mathfrak{m}}
\newcommand{\ba}{\mathfrak{a}}
\newcommand{\bb}{\mathfrak{b}}
\newcommand{\tld}{\widetilde }
\newcommand{\sI}{\scr{I}}
\newcommand{\mJ}{\mathcal{J}}
\DeclareMathOperator{\exc}{{exc}}
\newcommand{\sJ}{\scr{J}}
\input{xy}
\xyoption{all}
\usepackage{tikz}

\newcommand{\cf}{{\itshape cf.} }

\newcommand{\ie}{{\itshape i.e.} }

\DeclareMathOperator{\Newt}{Newt}
\DeclareMathOperator{\relint}{{relint}}

\usepackage{fullpage}

\usepackage{setspace}
\usepackage{hyperref}

\usepackage{enumerate}

\usepackage{graphicx}

\usepackage[all,cmtip]{xy}

\usepackage{verbatim}

\theoremstyle{theorem}
\newtheorem*{mainthm}{Main Theorem}
\newtheorem*{theorem*}{Theorem}
\newtheorem*{corollary*}{Corollary}

\renewcommand{\O}{\mathcal O}

 \title{Cartier modules on toric varieties}
\author{Jen-Chieh Hsiao, Karl Schwede, Wenliang Zhang}
\thanks{This research was initiated at the Commutative Algebra MRC held in June 2010.  Support for this meeting was provided by the NSF and AMS.}
\thanks{The first author was partially supported by the NSF grant DMS \#0901123}
\thanks{The second author was supported by an NSF postdoctoral fellowship and also by NSF grant DMS \#1064485.}
\thanks{The third author was partially supported by the NSF grant DMS \#1068946.}


\address{\vskip -.65cm \noindent  Jen-Chieh Hsiao: \sf Department of Mathematics,
  Purdue University.  West Lafayette, Indiana 47907}
\email{jhsiao@math.purdue.edu}

\address{\vskip -.65cm \noindent Karl Schwede: \sf Department of Mathematics,
  The Pennsylvania State University.  University Park, Pennsylvania 16802}
\email{schwede@math.psu.edu}

\address{\vskip -.65cm \noindent Wenliang Zhang: \sf Department of Mathematics,
  University of Michigan.  Ann Arbor, Michigan 48109-1109}
\email{wlzhang@umich.edu}

\begin{document}
\maketitle

\begin{abstract}
Assume that $X$ is an affine toric variety of characteristic $p > 0$. Let $\Delta$ be an effective toric $\bQ$-divisor such that $K_X+\Delta$ is $\mathbb{Q}$-Cartier with index not divisible by $p$ and let $\phi_{\Delta}:F^e_*\O_X\to \O_X$ be the toric map corresponding to $\Delta$. We identify all ideals $I$ of $\O_X$ with $\phi_{\Delta}(F^e_* I)=I$ combinatorially and also in terms of a log resolution (giving us a version of these ideals which can be defined in characteristic zero).  Moreover, given a toric ideal $\ba$, we identify all ideals $I$ fixed by the Cartier algebra generated by $\phi_{\Delta}$ and $\ba$; this answers a question by Manuel Blickle in the toric setting.
\end{abstract}

\section{Introduction}

Suppose that $R$ is a ring of characteristic $p > 0$ and $F : R \to R$ is the Frobenius map which we always assume is a finite map.  If $\phi : R \to R$ is a splitting of Frobenius, then there are finitely many ideals $I$ such that $\phi(I) \subseteq I$, see \cite{KumarMehtaFiniteness,SchwedeFAdjunction}.  These ideals are called \emph{$\phi$-compatible} and are an interesting and useful collection of objects to study in their own right (they are closely related to the characteristic zero notion of ``log canonical centers'').

Much more generally, suppose that $R$ is a reduced ring and $\phi : R \to R$ is an additive map that satisfies the condition $\phi(r^{p^e} \cdot x) = r \phi(x)$ for all $r, x \in R$ (for example, a splitting of Frobenius).  In \cite{BlickleBoeckleCartierModulesFiniteness}, M. Blickle and G. B\"ockle generalize the above mentioned finiteness results and show that there are finitely many $I \subseteq R$ such that $\phi(I) = I$ (such ideals we call \emph{$\phi$-fixed}).  In \cite{BlickleTestIdealsViaAlgebras}, M. Blickle generalized the notion of $\phi$-fixed ideals to include the additional data of an ideal $\ba$ to a formal real power $t \geq 0$, in fact, he generalized these fixed ideals in even greater settings.  However, the full finiteness results still remain illusive.  Explicitly, it is natural to study ideals $I$ such that
\[
\sum_{n >0} \phi^n \left( \overline{\ba^{t(p^{n e}-1)}}\cdot I \right)=I
\]
where $\phi^n$ denotes the $e$th iterate of $\phi$ and $\overline{J}$ denotes the integral closure of an ideal.  We call such ideals \emph{$\phi, \ba^t$-fixed}.

These $\phi$-fixed ideals are exactly the ideals for which one has certain Fujita-type global generation statements \cite[Section 6]{SchwedeACanonicalLinearSystem}.  In particular, identifying these ideals might be very useful in the problem of lifting sections for projective varieties in characteristic $p > 0$.
However, very few examples of the sets of $\phi$-fixed ideals, let alone $\phi, \ba^t$-fixed ideals, are known.  In this paper we compute these ideals in the toric setting.  In other words, $X = \Spec k[S]$ is an affine
toric variety, $\phi:F^{e}_*\O_X\to \O_X$ is a toric map and $\ba$ is a monomial ideal. Here $S = M \cap \sigma$ for some lattice $M$ and some rational convex polyhedral cone $\sigma$ in $M_{\bR}$. Because $\phi$ is a toric map, we can write $\phi(\blank) = \phi_c(x^{-w}\cdot \blank)$ for some $w
\in M$ where $\phi_c$ is the canonical splitting\footnote{The canonical splitting is the map which sends a monomial $x^m$ to $x^{m/p^e}$ if $m/p^e$ is an integer, and otherwise sends $x^m$ to zero.} of $F^e:k[S]\to k[S]$.  Finally let $t = a/b$ be a rational number such that $p$ does not divide $b$. Set $P = t \Newt(\ba) $ and $\sF = \{ \text{ faces of }P \}$. For $\tau \in \sF$, denote $I_{\tau} :=  \langle  x^v | v \in \relint({w \over 1-p^{e}}+\tau')\cap S \text{ for some }\tau' \supseteq \tau \rangle.$
Our main result is as follows:

\begin{mainthm} [Theorem \ref{cartier-algebra-fixed-ideal}]
$\sum_{n >0} \phi^{n} \left( \overline{\ba^{t(p^{ne}-1)}}\cdot I \right)=I $ if and only if $I = \sum_{\tau \in \sG} I_{\tau}$ for some $\sG \subseteq \sF$, where the sum is taken over all positive integers $n > 0$ such that $t(p^{n e}-1)$ is an integer.
\end{mainthm}

As an immediate corollary, we obtain:

\begin{corollary*}
Suppose that $(k[S], \phi, \ba^t)$ is as above.  Then there are only finitely many ideals $I$ such that $\sum_{n >0} \phi^n \left( \overline{\ba^{t(p^{n e}-1)}}\cdot I \right)=I $ where we again sum over $n$ such that $t(p^{ne} - 1)$ is an integer.
\end{corollary*}

This answers a question by Manuel Blickle (\cite[Question 5.4]{BlickleTestIdealsViaAlgebras}) in the toric setting.  Even in the case when $t=0$ (we define $J^0=R$ for any ideal $J$) or $\ba=R$, i.e. we have a pair $(R,\phi)$ instead of a triple $(R,\phi,\ba^t)$, our result identifies all $\phi$-fixed ideals.  We illustrate our Main Theorem in this special case below; in this case $\sigma = \Newt(\ba)$.

In the following diagram, circles represent the monomials of the semi-group ring $k[S]$ and solid lines represent the boundaries of $\sigma$.  Given $\phi$ as above, we consider the vector $w \over 1- p^e $.

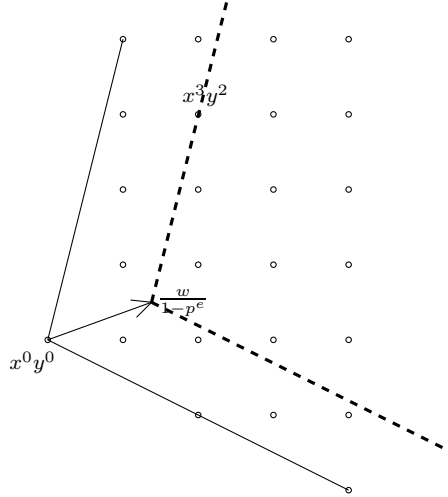
\begin{center}
{\scriptsize
\begin{tikzpicture}
\draw (0,0) -- (4,-2) (1,4)-- (0,0);
\draw (-0.2, -0.3) node{$x^0y^0$};
\draw (0,0) circle (0.3ex);
\draw (1,0) circle (0.3ex);
\draw (1,1) circle (0.3ex);
\draw (1,2) circle (0.3ex);
\draw (1,3) circle (0.3ex);
\draw (1,4) circle (0.3ex);
\draw (2,-1) circle (0.3ex);
\draw (2,0) circle (0.3ex);
\draw (2,1) circle (0.3ex);
\draw (2,2) circle (0.3ex);
\draw (2,3) circle (0.3ex);
\draw (2,4) circle (0.3ex);
\draw (3,-1) circle (0.3ex);
\draw (3,0) circle (0.3ex);
\draw (3,1) circle (0.3ex);
\draw (3,2) circle (0.3ex);
\draw (3,3) circle (0.3ex);
\draw (3,4) circle (0.3ex);
\draw (4,-2) circle (0.3ex);
\draw (4,-1) circle (0.3ex);
\draw (4,0) circle (0.3ex);
\draw (4,1) circle (0.3ex);
\draw (4,2) circle (0.3ex);
\draw (4,3) circle (0.3ex);
\draw (4,4) circle (0.3ex);
\draw (0, 0) -- (1.38, 0.5);
\draw (1.38, 0.5) -- (1.18, 0.3);
\draw (1.38, 0.5) -- (1.10, 0.53);
\draw (1.8, 0.5) node{$w \over 1- p^e $};
\draw (2.1, 3.25) node{$x^3y^2$};
\begin{scope}[very thick,dashed]
\draw (1.38, 0.5) -- (5.38, -1.5);
\draw (1.38, 0.5) -- (2.38, 4.5);
\end{scope}
\end{tikzpicture}
}
\end{center}

The $\phi$-fixed ideals will each be generated by monomials contained in the interior or boundary of the above  ``dotted'' region and we can explicitly identify them pictorially.
Explicitly, as our main theorem says, each of the $\phi$-fixed ideals will be generated by all monomials contained in one of the following shaded regions (in each region, the open circle corresponds to the point $w \over 1 - p^e$).

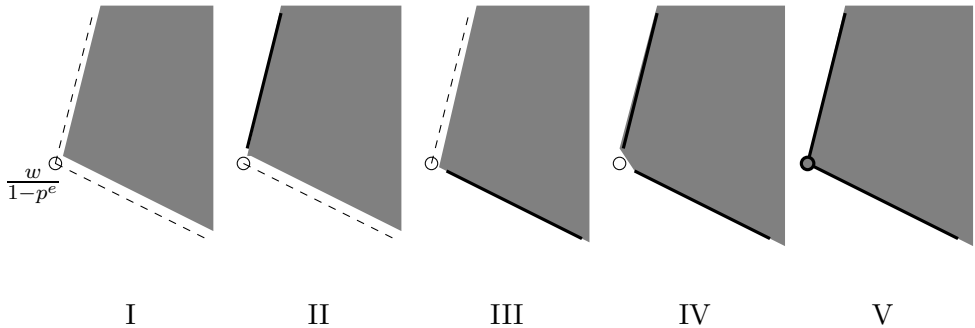
\begin{center}
\begin{tikzpicture}
\begin{scope}[dashed]
\draw (0, 0) -- (2, -1) (0,0) -- (0.5, 2) ;
\end{scope}
\draw (0,0) circle (0.5ex);
\fill[gray] (0.1, 0.1) -- (2.1, -0.9) -- (2.1, 2.1) -- (0.6, 2.1) -- cycle;
\draw (1, -2) node{I};
\draw (-0.3, -0.3) node{$w \over 1- p^e $};

\fill[gray] (2.6, 0.1) -- (4.6, -0.9) -- (4.6, 2.1) -- (3.0, 2.1) -- (2.55, 0.1) -- cycle;
\begin{scope}[dashed]
\draw (2.5, 0) -- (4.5, -1) ;
\end{scope}
\begin{scope}[very thick]
\draw (2.55,0.2) -- (3, 2);
\end{scope}
\draw (2.5,0) circle (0.5ex);
\draw (3.5, -2) node{II};

\fill[gray] (5.1, -0.05) -- (7.1, -1.05) -- (7.1, 2.1) -- (5.6, 2.1) -- cycle;
\begin{scope}[dashed]
\draw (5,0) -- (5.5, 2) ;
\end{scope}
\begin{scope}[very thick]
\draw (5.2, -0.1) -- (7, -1) ;
\end{scope}
\draw (5,0) circle (0.5ex);
\draw (6, -2) node{III};

\fill[gray] (7.7, -0.1) -- (9.7, -1.1) -- (9.7, 2.1) -- (8.0, 2.1) -- (7.5, 0.2) -- cycle;
\begin{scope}[very thick]
\draw (7.55,0.2) -- (8, 2);
\draw (7.7, -0.1) -- (9.5, -1) ;
\end{scope}
\draw (7.5,0) circle (0.5ex);
\draw (8.5, -2) node{IV};

\fill[gray] (10, 0) -- (12.2, -1.1) -- (12.2, 2.1) -- (10.5, 2.1) -- cycle;
\begin{scope}[very thick]
\draw (10,0) -- (10.5, 2);
\draw (10, 0) -- (12, -1) ;
\draw[fill=gray] (10,0) circle (0.5ex);
\draw (11, -2) node{V};
\end{scope}
\end{tikzpicture}
\end{center}
While each of the ideals associated with these different bodies are potentially different, in many cases (depending on the particular $w$), they are the same.

As we have already noted, in the case that $\phi$ is a Frobenius splitting, the $\phi$-fixed ideals are closely related to log canonical centers (a notion defined by using a resolution of singularities).  It is thus natural to ask if these ideals $I$ such that $\phi(I) = I$ are also related to a notion defined using a resolution of singularities.  At least in the toric setting, we identify a class of ideals, defined using a resolution of singularities which coincides with the $\phi$-fixed ideals $I$.
Our main result on relating $\phi$-fixed ideals and resolutions of singularities is the following.

\begin{theorem*}[Theorem \ref{adjoint-ideals-are-fixed}]
Let $X$ be an affine toric variety of characteristic $p>0$ and let $\Delta$ be an effective torus-invariant $\bQ$-divisor on $X$ such that $K_X+\Delta$ is $\bQ$-Cartier. Let $\ba$ be a toric ideal and $t$ a nonnegative rational number which can be written without $p$ in its denominator. Let $\phi:F^e_*\O_X\to \O_X$ be the toric map corresponding to $\Delta$. Then an ideal $I\subset \O_X$ satisfies $\sum_{n >0} \phi^n \left( \overline{\ba^{t(p^{ne}-1)}}\cdot I \right)=I $ if and only if there exists a toric log resolution $\pi:X'\to X$ of $(X,\Delta,\ba)$ such that $\ba\cdot \O_{X'}=\O_{X'}(-G)$ and an effective divisor $E$ on $X'$ with $\pi(E)\subset \Supp\Delta\cup \Sing X \cup V(\ba)$ such that
\[I=\pi_*\O_{X'}(\lceil K_{X'}-\pi^*(K_X+\Delta) -tG +\varepsilon E\rceil)\ for\ 1\gg \varepsilon >0. \]
\end{theorem*}

This result perhaps should not be unexpected, we explain the motivation for this result in the special case that $\ba = R$.  Note that $\phi$ extends to a map $\phi' : F^{e}_* \sM \to \sM$ where $\sM = \O_{X'}(\lceil K_{X'}-\pi^*(K_X+\Delta) + E\rceil)$ for every effective $\bQ$-divisor $E$ on $X'$, \cite[Proof \#2 of Theorem 3.3]{HaraWatanabeFRegFPure}, \cite[Theorem 6.7]{SchwedeCentersOfFPurity}.  Choose now $1 \gg \varepsilon > 0$, and it follows that the fractional ideal
\[
\sI = \O_{X'}(\lceil K_{X'}-\pi^*(K_X+\Delta) + \varepsilon E\rceil)
\]
is $\phi'$-fixed.  The ideals discussed in Theorem \ref{adjoint-ideals-are-fixed} are thus exactly the push-forwards of $\phi'$-fixed fractional ideals whose pushforwards are honest (non-fractional) ideals of $\O_X$.

\vskip 6pt
\noindent
{\it Acknowledgements: }

The authors began work on this project during the Commutative Algebra MRC held in June 2010 in Snowbird Utah, organized by David Eisenbud, Craig Huneke, Mircea Musta{\c{t}}{\u{a}}, and Claudia Polini.  We would like to thank Mircea Musta{\c{t}}{\u{a}}, Lance Miller, and Julian Chan for numerous valuable discussions and encouragement at the MRC.  We also thank Manuel Blickle several discussions related to these ideas and in particular discussions clarifying Example \ref{exInfinitelyManyIdealsEventually}.  Finally we thank the referee for some quite helpful suggestions.

\section{Preliminaries for fixed ideals}
\label{secPreliminariesForFixed}

Throughout this paper, all rings are Noetherian and excellent and possess dualizing complexes and all schemes are separated.  Mostly we will work in the setting of toric varieties where all these conditions are automatic.

In this section, $R$ is a normal $F$-finite domain of characteristic $p > 0$ which admits an additive map $\phi:  R \rightarrow R$ satisfying the relation that
$\phi(r^{p^e}x) = r \phi(x)$ for all $r,x \in R$. Such a $\phi$ is called a $p^{-e}$-linear map.
Typical examples of such maps are the maps that split the Frobenius map $F: R \rightarrow R$.  It can be difficult to distinguish the source and target of this map.  Thus for any $R$-module $M$, we use $F^e_* M$ to denote the $R$-module isomorphic to $R$ as an Abelian group but with the $R$-action, $r.m = r^{p^e} m$.  From this perspective, a $p^{-e}$-linear map is simply an $R$-linear map $F^e_* R \rightarrow R$.  For the rest of the paper, all $p^{-e}$-linear maps $\phi : R \to R$ will be written as $R$-linear maps $\phi : F^e_* R \to R$.

Once a map $\phi : F^e_* R \to R$ has been fixed, $R$ fits into the theory of Cartier modules developed by Blickle and B\"ockle \cite{BlickleBoeckleCartierModulesFiniteness}.
One of the main theorems in \cite{BlickleBoeckleCartierModulesFiniteness} guarantees that there are only finitely many ideals $I$ of $R$ satisfying $\phi(F^e_* I)=I$.
We will call such an $I$ a \emph{$\phi$-fixed ideal} (according to the terminology of \cite{BlickleTestIdealsViaAlgebras}, these ideals are also called \emph{$F$-pure Cartier-submodules of $(R, \phi)$}).  Moreover, the theory of Cartier modules in \cite{BlickleBoeckleCartierModulesFiniteness} was generalized in \cite{BlickleTestIdealsViaAlgebras} which we will review briefly.

\subsection{Cartier Algebras}

\begin{definition}[\cite{SchwedeTestIdealsInNonQGor}, \cf \cite{BlickleTestIdealsViaAlgebras}]
Let $R$ be commutative Noetherian ring of characteristic $p>0$. An \emph{algebra of $p^{-e}$-linear maps on $R$}, or simply a \emph{Cartier algebra on $R$} is an $\mathbb{N}$-graded ring $\mathcal{C}:=\bigoplus_{e\geq 0}\mathcal{C}_e$ such that each $\phi \in \mathcal{C}_e$ is a map $\phi : F^e_* R \to R$.  Furthermore, multiplication of elements of $\mathcal{C}$ corresponds to composition of maps.  In other words, for $\phi \in \mathcal{C}_e$ and $\psi \in \mathcal{C}_d$, we define:
\[
\phi \cdot \psi := \phi \circ (F^e_* \psi).
\]
We call the pair $(R, \mathcal{C})$ \emph{$F$-pure} if there exists a surjective $\phi \in \mathcal{C}_e$ for some $e > 0$.
\end{definition}

If one uses $\End_e(R)$ to denote $\Hom_R(F^e_*R,R)$, then $\mathcal{C}_R:=\bigoplus_e\End_e(R)$ will be an example of an $R$-Cartier-algebra. In this paper, we will focus on the subring determined by an ideal $\ba$, a nonnegative real number $t$ and a $p^{-e}$-linear map $\phi$, denoted by $\mathcal{C}^{\phi, \ba^t}$, \ie
\[
\mathcal{C}^{\phi, \ba^t}:=  \bigoplus_n \phi^n\cdot F^{ne}_* \overline{\ba^{t(p^{ne}-1)}} = \bigoplus_{n \geq 0} \mathcal{C}^{\phi, \ba^t}_{n},
\]
where the sum is taken over all $n$ such that $t(p^{ne}-1)$ is an integer.  As mentioned before, $\overline{J}$ is used to denote the integral closure of $J$, see \cite{HunekeSwansonIntegralClosure}.  If $\ba = R$ or $t = 0$, then we use $\mathcal{C}^{\phi}$ to denote $\mathcal{C}^{\phi, \ba^t}$.

Furthermore, we define
\[
\mathcal{C}^{\phi, \ba^t}_{+} := \bigoplus_{n > 0} \mathcal{C}^{\phi, \ba^t}_n.
\]
Finally, we state the definition of our main object of study in this paper.

\begin{definition}
If an ideal $I$ satisfies $\mathcal{C}^{\phi, \ba^t}_+(I)=I$, \ie $\sum_{n > 0} \phi^{n} (\overline{F^{ne}_* \ba^{t(p^{ne}-1)}}\cdot I) = I$, then $I$ is called $\mathcal{C}^{\phi, \ba^t}$-fixed. According to the terminology of \cite{BlickleTestIdealsViaAlgebras}, these ideals are also called \emph{$F$-pure Cartier-submodules of the triple $(R, \phi,\ba^t)$}.
\end{definition}

Note that when either $t=0$ or $\ba=R$, this Cartier-algebra $\mathcal{C}^{\phi, \ba^t}$ is the sub-Cartier-algebra of $\mathcal{C}_R$ generated by $\phi$; we will denote this sub-Cartier-algebra by $\mathcal{C}^{\phi}$. In this special case, one of the main theorems in \cite{BlickleBoeckleCartierModulesFiniteness} guarantees that there are only finitely many ideals $I$ of $R$ satisfying $\phi(F^l_*I)=I$ ( \ie $\mathcal{C}^{\phi}_+(I)=I$). However, the finiteness of $\mathcal{C}^{\phi, \ba^t}$-fixed ideals was left as an open question in \cite[Question 5.4]{BlickleTestIdealsViaAlgebras}, to which our main theorem (Theorem \ref{cartier-algebra-fixed-ideal}) gives a partial positive answer.

\subsection{Fixed ideals}

Fixed ideals have appeared throughout commutative algebra and representation theory.  Indeed, if $\phi : F_* R \to R$ is a splitting of Frobenius, then a \emph{compatibly-$\phi$-split ideal} $J \subseteq R$ is an ideal such that $\phi(F_* J) = J$.  Alternately, note there is always an $R$-linear map $\phi^e : F_*^e \omega_R \to \omega_R$ for each $e > 0$ -- the Grothendieck trace map.  If $R$ is Gorenstein and local, then $\omega_R \cong R$ and thus we obtain a a canonical $p^{-e}$-linear map $\phi^e : F^e_* R \rightarrow R$.
In this case, the smallest non-zero ideal $J$ such that $\phi^e(F^e_*J)=J$ is the (big) test ideal of $R$, \cite{SchwedeFAdjunction}.  Furthermore, if $\phi$ is surjective, the largest such proper ideal is the splitting prime of Aberbach and Enescu, \cite{AberbachEnescuStructureOfFPure}.  In this section, we develop the theory of $\mathcal{C}^{\phi, \ba^t}$-fixed ideals.

We mention the following results about the set of $\phi$-fixed ideals which we will need.
\begin{proposition}
\label{propPropertiesOfFixedIdeals}
Suppose that $(R,\phi,\ba^t)$ is a triple as above with $R$ an $F$-finite domain\footnote{We make this hypothesis only for simplicity, most of what follows below can be generalized outside of this setting with minimal work.}.
\begin{itemize}
\item[(i)]  The set of $\mathcal{C}^{\phi, \ba^t}$-fixed ideals of $R$ is closed under sum.
\item[(ii)]  If $\phi$ is surjective, then the set of $\phi$-fixed ideals is closed under intersection.
\item[(iii)]  If $\phi$ is surjective, then any $\phi$-fixed ideal is a radical ideal.
\item[(iv)]  If $R$ is a normal domain and $\phi$ corresponds to a $\bQ$-divisor $\Delta_{\phi}$ as in subsection \ref{subsecnRelationBetweenMapsAndDivisors} below, then the test ideal\footnote{The reader may take this to be the definition of the test ideal if they are not already familiar with it.} $\tau(R, \Delta_{\phi},\ba^t)$ is the unique smallest non-zero $\mathcal{C}^{\phi, \ba^t}$-fixed ideal of $R$.
\item[(v)]  There are finitely many $\phi$-fixed ideals\footnote{As our main result shows, in the toric setting there are also finitely many $\mathcal{C}^{\phi, \ba^t}$-fix ideals.}.
\item[(vi)]  For an element $d \in R$, define a new map $\psi(\blank) = \phi(F^e_*d^{p^e - 1} \cdot \blank)$.  Then an ideal $J \subseteq R$ is $\mathcal{C}^{\phi, \ba^t}$-fixed if and only if $d J$ is $\mathcal{C}^{\psi,\ba^t}$-fixed.
\end{itemize}
\end{proposition}
\begin{proof}
Part (i) is trivial from the definition.  Part (ii) follows from the observation that if $\phi$ is surjective, then any ideal $J$ satisfying the condition $\phi( F^e_*J) \subseteq J$ is automatically $\phi$-fixed, and the set of these ideals is closed under intersection.  Part (iii) is again easy.  Part (iv) is an easy exercise based upon \cite{SchwedeTestIdealsInNonQGor}.
Part (v) is, as mentioned, one of the main results of \cite{BlickleBoeckleCartierModulesFiniteness}.

To prove (vi), suppose first that $J$ is $\mathcal{C}^{\phi, \ba^t}$-fixed so that $\sum_{n>0}\phi^n(\overline{\ba^{t(p^{ne}-1)}} J) = J$.  Then
\begin{align}
\sum_{n>0}\psi^n(\overline{\ba^{t(p^{ne}-1)}} dJ) &= \sum_{n>0}\phi^n(F^{ne}_*d^{p^{ne}-1}\overline{\ba^{t(p^{ne}-1)}}d J)\notag\\
&=\sum_{n>0}\phi^n(F^{ne}_*d^{p^{ne}}\overline{\ba^{t(p^{ne}-1)}} J)\notag\\
&=d\sum_{n>0}\phi^n(\overline{\ba^{t(p^{ne}-1)}}J)\notag\\
&=dJ\notag
\end{align}
The converse statement merely reverses this.
\end{proof}

\begin{remark}
We will see in the toric setting that the set of $\phi$-fixed ideals is closed under intersection for any $\phi$.  It would be interesting to discover if this holds more generally.
\end{remark}

\begin{remark}
\label{remEarlyInfinitelyManyFixedIdeals}
Suppose that $R$ is a normal domain and $\phi : F^e_* R \to R$ is an $R$-linear map as in Proposition \ref{propPropertiesOfFixedIdeals}.  One can always extend $\phi$ to a map $\bar{\phi} : F^e_* K(R) \to K(R)$ where $K(R)$ is the fraction field of $R$.  While it is true that there are only finitely many $\phi$-fixed ideals of $R$, in general there are infinitely many $\phi$-fixed fractional ideals of $K(R)$.  For example, consider the ring $R = k[x]$ with the canonical splitting $\phi_c : F_* R \to R$.  Then the fractional ideal generated by $1 \over x+1$ is $\phi_c$-fixed.  Indeed, a generating set of this fractional ideal over $R^p$ is $\left\{ {x^i \over x+1} \right\}_{0 \leq i \leq p-1}$.  Notice that
\[
\xymatrix{
\phi_c\left({x^i \over x+1}\right) = \phi_c\left({(x+1)^{p-1} x^i \over (x+1)^p}\right) = {1 \over x+1} \phi_c\left(x^{p-1 + i} + {p-1 \choose 1 } x^{p-1 + i -1} + \dots + x^i\right) = b{x^{\lceil i / p \rceil} \over x+1}
}
\]
for some constant $b$, which proves the claim.  Clearly the fractional ideal generated by $1 \over x+1$ is not toric.  Of course, the same statement holds for the fractional ideal generated by $1 \over x+\lambda$ for any $\lambda \in k$ as well as for many other ideals.

\end{remark}

We now give a method for constructing $\phi$-fixed ideals.  While we will not use it directly, an analog of this result for ideals defined using resolution of singularities is the key observation which allows us to characterize $\phi$-fixed ideals via a resolution, compare with Proposition \ref{PropHowToConstructInterAdj}.
We first recall that given $\phi : F^e_* R \to R$, we can compose $\phi$ with itself to obtain a map $\phi^2 = \phi \circ F^e_* \phi : F^{2e}_* R \to R$ and similarly construct $\phi^n$ for any positive integer $n$.  While \emph{a priori}, we may have an infinite descending chain of ideals $\phi^n(F^{ne}_* R) \supseteq \phi^{n+1}(F^{(n+1)e}_* R) \supseteq \dots$; it is a theorem of Gabber that this chain eventually stabilizes, see \cite{Gabber.tStruc} (also see  \cite{BlickleTestIdealsViaAlgebras} for a generalization and \cite{HartshorneSpeiserLocalCohomologyInCharacteristicP} for the local dual statement in the geometric setting).
\begin{definition}
The stable image $\phi^n(F^{ne}_* R) = \phi^{n+1}(F^{(n+1)e}_* R) = \dots$ will be denoted by $S(\phi) \subseteq R$.  it is automatically $\phi$-fixed (and it is by construction, the largest $\phi$-fixed ideal).
\end{definition}
In \cite{FujinoSchwedeTakagiSupplements}, $S(\phi)$ was denoted by $\sigma(\phi)$.  We avoid this notation because we already are utilizing $\sigma$.

\begin{proposition}
\label{propWayToConstructFixedIdeals}
Suppose that $R$ is a domain and $d \in R$ is a non-zero element.  Then for any $n \gg 0$, if we define a map $\psi_n : F^{ne}_* R \to R$ by the formula $\psi_n(\blank) = \phi^n(F^{ne}_* d \cdot \blank)$, we have that $S(\psi_n)$ is $\phi$-fixed.
\end{proposition}
\begin{proof}
We first claim that  $S(\psi_n) \subseteq S(\psi_{n+1})$ for any $n$ (compare also with \cite[Proposition 14.10(1)]{FujinoSchwedeTakagiSupplements}).  To prove this claim, first notice that for any $\alpha : F^e_* R \to R$, $S(\alpha) = S(\alpha^m)$.  Thus, in order to show the claim for $\psi$, it suffices show that $\psi_n^{n+1}(F^{(n+1)ne}_* J) \subseteq \psi_{n+1}^n(F^{(n+1)ne}_* J)$ for every ideal $J$.  However, because
\[
1 + p^{ne} + \dots + p^{n^2e} = {p^{(n+1)ne} - 1 \over p^{ne} - 1 } \geq { p^{n(n+1)e} - 1 \over p^{(n+1)e} - 1 }= 1 +p^{(n+1)e} + \dots + p^{(n-1)(n+1)e}
\]
we obtain that
\[
\begin{split}
\psi_n^{n+1}(F^{(n+1)ne}_* J) = \phi^{n(n+1)}(F^{(n+1)ne}_* d^{1 + p^{ne} + \dots + p^{n^2e} } J) \\
\subseteq \phi^{n(n+1)}(F^{(n+1)ne}_* d^{ 1 + p^{(n+1)e} + \dots + p^{ (n-1)(n+1)e }} J) = \psi_{n+1}^{n}(F^{(n+1)ne}_* J)
\end{split}
\]
which proves the claim.

We choose $n$ which stabilizes this chain $S(\psi_n) = S(\psi_{n+1})$.  Suppose now $m > 0$ is such that $S(\psi_n) = \psi_n^m(F^{m(ne)}_* R)$.
Then
\[
\begin{split}
\phi(S(\psi_n))\\
 = \phi(F^e_* \psi_n^{m+1}(F^{(m+1)ne}_* R))\\
 = \psi_{n+1}\left(F^{(n+1)e}_* ( \psi_n^m(F^{mne}_* R) ) \right) \\
= \psi_{n+1}\left( F^{(n+1)e}_* S(\psi_n) \right)\\
= \psi_{n+1}\left( F^{(n+1)e}_* S(\psi_{n+1}) \right) \\
= S(\psi_{n+1})\\
= S(\psi_n)
\end{split}
\]
which proves the proposition.
\end{proof}

\subsection{The relation between $\phi$ and $\bQ$-divisors.}
\label{subsecnRelationBetweenMapsAndDivisors}

Later, we will relate the $\phi$-fixed ideals with the ideals coming from a resolution of singularities in characteristic $0$ (e.g. multiplier ideals, Fujino's non-LC ideal, and the
ideals defining arbitrary unions of log canonical centers). This relation comes from a correspondence between pairs $(X=\Spec R,\Delta)$ and certain $p^{-e}$-linear maps $\phi:R
\rightarrow R$ in the theory of $F$-singularities. The reader is referred to \cite{SchwedeFAdjunction} for a detailed account of this correspondence, also see \cite{SchwedeTuckerTestIdealSurvey}.

Suppose that $(X, \Delta)$ is a pair where $X$ is a variety of finite type over an $F$-finite field $k$ such that $K_X + \Delta$ is $\bQ$-Cartier with index not divisible by
$p > 0$.  Further suppose that $\Delta$ is an effective $\bQ$-divisor such that $(p^e-1)\Delta$ is integral and $(p^e-1)(K_X+\Delta)$ is Cartier for some $e$.
Then there is a bijection of sets:
\begin{equation*}
\left\{
\begin{aligned}
&\text{Effective $\bQ$-divisors $\Delta$ on $X$ such }\\
&\text{that $(p^e -1)(K_X + \Delta)$ is Cartier}
\end{aligned}
\right\}
\longleftrightarrow
\left\{
\begin{aligned}
&\text{Line bundles $\sL$ and non-zero} \\
&\text{elements of $\Hom_{\O_X} (F^e_* \sL , \O_X )$}
\end{aligned}
\right\}
\bigg/ \sim
\end{equation*}
The equivalence relation on the right side identifies two maps $\phi_1 : F^e_* \sL_1 \rightarrow \O_X$ and
$\phi_2 : F^e_* \sL_2 \rightarrow \O_X$ if there is an isomorphism $\gamma : \sL_1 \rightarrow \sL_2$ and a commutative diagram:
\[\begin{CD}
F^e_* \sL_1  @>\phi_1 >> \O_X \\
@V F^e_* \gamma VV      @VV \id V \\
F^e_* \sL_2  @>\phi_2 >> \O_X
\end{CD}.\]
Given $\Delta$, set $\sL = \O_X ((1 -p^e )(K_X + \Delta))$. Then observe that
          \[  F^e_* \O_X ((p^e - 1)\Delta) \cong  F^e_* \sHom_{\O_X} (\sL , \O_X ((1 - p^e )K_X )) \cong \sHom_{\O_X} (F^e_* \sL , \O_X ).
          \]
The choice of a section $\eta \in \O_X ((p^e - 1)\Delta)$ corresponding to $(p^e-1)\Delta$ thus gives a map $\phi_{\Delta} : F^e_* \sL \to \O_X$.  The choice depends on various isomorphisms selected, but this is harmless for our purposes.

For the converse direction, an element $\phi \in \sHom_{\O_X}(F^e_* \sL, \O_X) \cong F^e_* \sL^{-1}((1-p^e)K_X)$ determines an $F^e_* \O_X$-linear map
\[ F^e_* \O_X \xrightarrow{1 \mapsto \phi} \sHom_{\O_X}(F^e_* \sL, \O_X) \xrightarrow{\sim} F^e_* \sL^{-1}((1-p^e)K_X) \]
which corresponds to an effective Weil divisor $D$ such that $\O_X(D) \cong \sL^{-1}((1-p^e)K_X)$. Set $\Delta_{\phi} = {1 \over p^e-1} D$.

\begin{remark}
\label{remTweakMapTweaksDivisor}
Explicitly, suppose $\phi : F^e_* R \to R$ is an $R$-linear map and $d \in R$.  Define a new map $\psi(\blank) = \phi(F^e_* d \cdot \blank)$.  Then $D_{\psi} = D_{\phi} + {1 \over p^e - 1} \Div(d)$.
\end{remark}

\begin{definition}
Suppose that $\phi : F^e_* R \to R$ corresponds to a divisor $\Delta$.  Then we define $S(R, \Delta)$ to be $S(\phi)$ where $S(\phi)$ is defined in the paragraph before Proposition \ref{propWayToConstructFixedIdeals}.
\end{definition}

We illustrate the above construction by the case where $X$ is an affine toric variety (which recall has trivial Picard group, so that every line bundle is isomorphic to $\O_X$)
Let $M$ be a lattice and $\sigma$ be a rational convex polyhedral cone in $M_{\bR}$. Set $R = k [\sigma \cap M]$.
Suppose $\Delta$ is an effective $\bQ$-divisor on $X = \Spec R$ as above. Then we can write
\[ (1-p^e)(K_X+\Delta) = \Div_X(x^w) \]
for some $w \in M$.
Then a map $\phi_{\Delta}$ corresponding to $\Delta$ can be expressed as
\[\phi_{\Delta}(\blank) = \phi_c(x^{-w} \cdot \blank)\]
where $\phi_c$ is the canonical splitting on $R$ defined by
\[\phi_c(x^v)= \begin{cases}
x^{v\over p^e}& \text{if ${v\over p^e} \in M$}, \\
0 & \text{otherwise.}
\end{cases}
\]

Let us explain these claims carefully since this identification is critical for what follows.
Our first claim is that the map $\phi_c$ corresponds to the torus invariant divisor $\Delta_c = -K_X$.  It is sufficient to show that $\phi_c$ fixes every height-one prime torus invariant ideal (which implies that $\Delta_{\phi_c}$ contains each torus invariant divisor as a component) and does not fix any other height-one ideal (which implies that $\Delta_{\phi_c}$ is toric).  While both these statements are well-known to experts, we point out that the first statement is simply \cite[Proposition 3.2]{PayneFrobeniusSplitToric} while the second is a very special case of the proof of Proposition \ref{ismonomial} below.   So now suppose that $\Delta$ is a torus invariant divisor such that $(1-p^e)(K_X + \Delta)$ is Cartier and thus is equal to $\Div_X(x^w)$.  Therefore \[
(1-p^e)(K_X + \Delta) = \Div_X(x^w) + 0 = \Div_X(x^w) + (1-p^e)(K_X + (-K_X))
\]
Dividing through by $(1-p^e)$ gives us $\Delta = (-K_X) + {1 \over p^e - 1}\Div_X(x^{-w})$.  However, it is easy to see that given any map $\beta : F^e_* \O_X \to \O_X$ and any element $d \in F^e_* \Frac R$, the map $\alpha(\blank) = \beta(d \cdot \blank)$, if it is indeed a map $\alpha : F^e_* \O_X \to \O_X$, has associated divisor $\Delta_{\alpha} = \Delta_{\beta} + {1 \over p^e - 1} \Div_X(d)$.  In other words, the map $\phi_{\Delta}$ as described above does indeed correspond to $\Delta$.

\subsection*{Notation.} For an $\bR$-Weil divisor $D=\sum^r_{j=1}d_jD_j$ such that $D_j$'s are distinct prime Weil divisors, we define
\[\lceil D\rceil=\sum^r_{j=1}\lceil d_j\rceil D_j\ {\rm and}\ \lfloor D\rfloor =\sum^r_{j=1}\lfloor d_j\rfloor D_j,\]
where for each real number $x$, the round-up (resp. round-down) $\lceil x\rceil$ (resp. $\lfloor x\rfloor$) denotes the integer defined by $x\leq \lceil x\rceil <x+1$ (resp. $x_1<\lfloor x\rfloor\leq x$). We also define
\[D^{\geq k}=\sum_{d_j\geq k}d_jD_j.\]

\section{$\mathcal{C}^{\phi, \ba^t}$-fixed ideals on Toric Varieties}
\label{secTestModulesonToricVarieties}

In this section, we compute the fixed ideals of certain Cartier algebras (in the sense of Blickle) on toric varieties.

We fix $M$ to be a lattice and $\sigma$ to be a rational convex polyhedral cone in $M_{\bR}$, set $S=\sigma \cap M$, $R=k[S]$, and $X=\Spec(k[S])$. Let $d$ denote the dimension of $X$, and let $\ba$ be a monomial ideal on $X$. Let $\Delta$ be a toric divisor on $X$ such that $(1-p^e)(K_X+\Delta)=\Div_X(x^w)$ for some positive integer $e$ and some element $w\in M$.

Denote $ \phi_c: F^e_* \O_X \rightarrow \O_X$ the canonical splitting on $X$.
Consider a $p^{-e}$-linear map $\phi_{\Delta}(\blank) = \phi_c(x^{-w} \blank)$ ($w$ is determined by $\Delta$; when $\Delta$ is clear we will simply write $\phi$), \ie
\[\phi(x^u)=\begin{cases}  x^{\frac{u-w}{p^e}} & {\rm if\ }\frac{u-w}{p^e}\in M\\ 0 & {\rm otherwise}\end{cases}\]

Given a rational number $t >0$ that can be written without $p$ in its denominator, we describe all the ideals $I$ fixed the Cartier algebra $\mathcal{C}^{\phi, \ba^t}$ generated by $\phi$ and $\ba^t$, \ie the ideals $I$ satisfying
\[ \sum_{n >0} \phi^n \left( \overline{\ba^{t(p^{ne}-1)}}\cdot I \right)=I \]
where the sum runs through all $n>0$ such that $t(p^{ne}-1)$ is an integer.

\begin{lemma}\label{propertiesNewt}
Let $R$ be a normal toric algebra, and let $I$, $J$  be monomial ideals of $R$.
Denote $\Newt(I)$ the Newton polytope of an ideal $I$.
Then
\begin{enumerate}
\item $\overline{I} = \langle x^v \mid v \in \Newt(I) \cap M \rangle$ \cite{FultonToric}.
\item $\Newt(I)+\Newt(J)=\Newt(IJ)$; in particular, $\overline{I^n} =
\langle x^v| v \in (n\Newt(I)) \cap M \rangle$ for any positive integer $n$. \item If $I = \langle x^g | g \in \Gamma \rangle$, then $\Newt(I) = \Conv(\Gamma)+\sigma$ where $\Conv(\Gamma)$ is the convex hull of $\Gamma$.
\end{enumerate}
\end{lemma}
\begin{proof} We simply prove {\it{(2)}} and {\it{(3)}}.

For {\it{(2)}}, by definition $\Newt(I)= \Conv \{a \in M | x^a \in I \}$.  Then note that $\Newt(I) + \Newt(J)$ is simply
\[
\{ (s_1 u_1 + \dots + s_l u_l) + (t_1 v_1 + \dots + t_m v_m)\, |\, l,m > 0, \Sigma s_i = 1, s_i \geq 0, x^{u_i} \in I, \Sigma t_j = 1, t_j \geq 0, x^{v_j} \in J\}
\]
By allowing repeats in the $u_i$ and $v_j$ and possibly making some $s_i, t_j = 0$, we may assume that $l = m$ and that $s_i = t_i$ for all $i$.  Thus $\Newt(I) + \Newt(J)$ is equal to
\[
\{ s_1 (u_1 + v_1) + \dots + s_l (u_l + v_l) \, |\, l > 0, \Sigma s_i = 1, s_i \geq 0, x^{u_i} \in I, x^{v_i} \in J\}
\]
which is clearly equal to $\Newt(IJ)$ as desired.

We now prove {\it{(3)}}.
Every point in $\Newt(I)$ is of the form $\sum n_i a_i$ for some $a_i \in M$ with $x^{a_i} \in I$, $n_i >0$, and $\sum n_i=1$.
For each $i$, since $x^{a_i} \in I$, there exist $g_i \in \Gamma$ and $s_i \in S$ such that $a_i = g_i +s_i$. Therefore,
$\sum n_i a_i = (\sum n_i g_i)+ (\sum n_is_i) \in \Conv(\Gamma) +\sigma$. So $\Newt(I) \subseteq \Conv(\Gamma) + \sigma$.
The proof of the other containment is similar to {\it (2)} above.
\end{proof}

\begin{proposition}\label{ismonomial}
Let $R,\phi,\ba^t$ be as above. If $I$ is $\mathcal{C}^{\phi,\ba^t}$-fixed, then $I$ is monomial and $\{v | x^v \in I\}   \subseteq {w \over 1-p^e} +t\Newt(\ba) $.
\end{proposition}

\begin{proof} For any element $h=\sum c_i x^{m_i} \in I$, fix any term $x^{m_i}$ and set $m=m_i$. Since
$$\sum_{n >0} \phi^n \left( \overline{\ba^{t(p^{en}-1)}}\cdot I \right) =I,$$ where the sum is taken over all integers $n>0$ such that $t(p^{en}-1)$ is an integer, there exist $n' >0$, $a' \in \Newt( \ba^{t(p^{en}-1)})$, and $m' \in S$ such that $x^{m'}$ is a term of some element in $I$ and that $\phi^{n'}( x^{a'}\cdot x^{m'}) = x^m$. Then
\[ m = {a'+m' - {p^{en'}-1 \over p^e-1} w \over p^{en'}}.\]

Repeating the same process $k$ times for $m'$ in the place of $m$, we can write
\[ m= {  \left( \sum^{k-1}_{i=1} p^{e\left(\sum^k_{j=i+1} n^{(j)}\right)} a^{(i)} \right) +a^{(k)}+ m^{(k)} - {p^{e\left(\sum^k_{j=1} n^{(j)}\right)}-1 \over p^e-1}w \over p^{e(\sum^k_{j=1} n^{(j)})}},
\]
where $a^{(i)} \in \Newt \left( \ba^{t\left(p^{(en^{(i)})}-1\right)} \right)$ and $m^{(k)} \in S$.

By Lemma \ref{propertiesNewt} and the fact that $\Newt(I)+S=\Newt(I)$, we have
\[
\begin{aligned}
&mp^{e\sum^k_{j=1} n^{(j)}}+ {p^{e\sum^k_{j=1} n^{(j)}}-1 \over p^e-1}w\\
=&\left( \sum^{k-1}_{i=1} p^{e\left(\sum^k_{j=i+1} n^{(j)}\right)} a^{(i)} \right) +a^{(k)}+ m^{(k)}\\
\in& \left( \sum^{k-1}_{i=1} p^{e\left(\sum^k_{j=i+1} n^{(j)}\right)}
 \Newt \left( \ba^{t\left(p^{(en^{(i)})} - 1\right)} \right) \right) + \Newt \left( \ba^{t\left(p^{en^{(k)}-1}\right)} \right)\\
 = & t \Bigg( \Big( \sum_{i=1}^{k-1} p^{e\left(\sum^k_{j=i+1} n^{(j)}\right)}  {\left(p^{(en^{(i)})}-1\right)} \Big) + (p^{(en^{(k)})} - 1)\Bigg) \cdot \Newt(\ba)\\
  = & t \Bigg( \Big( \sum_{i=1}^{k-1} \big(p^{e\left(\sum^k_{j=i} n^{(j)}\right)} - p^{e\left(\sum^k_{j=i+1} n^{(j)}\right)}\big)  \Big) + (p^{(en^{(k)})} - 1)\Bigg) \cdot \Newt(\ba)\\
  = & t \Bigg( \left(p^{e\left(\sum^k_{j=1} n^{(j)}\right)} - p^{(e n^{(k)})}\right) + (p^{(en^{(k)})} - 1)\Bigg) \cdot \Newt(\ba)\\
  =& t\left(p^{e(\sum^k_{j=1} n^{(j)})}-1\right) \cdot \Newt(\ba)\\
 \end{aligned}
 \]
Dividing by $p^{e\sum^k_{j=1} n^{(j)}}$ and letting $k$ go to infinity, we see that $m \in {w \over 1-p^e} + t \Newt(\ba)$.

Now for any $n>0$ such that $t(p^{en}-1)$ is an integer, $(p^{en}-1)m + {p^{en} -1 \over p^e-1}w\in t(p^{en}-1) \Newt(\ba)$ and hence by Lemma \ref{propertiesNewt}
\[x^{(p^{en}-1)m + {p^{en} -1 \over p^e-1}w} \in  \overline{ \ba^{t(p^{en}-1)}}.\]
Notice that for any term $x^{m_j}$ in $h \in I$,  $\phi^n( x^{(p^{en}-1)m + {p^{en} -1 \over p^e-1}w} \cdot x^{m_j}) = x^{m + { m_j -m \over p^{en}}}.$ Therefore, $\phi^n( x^{(p^{en}-1)m-{p^{en} -1 \over p^e-1}w} \cdot h) \in I$ is a nonzero constant multiple of $x^{m}$ for $n \gg 0$. So $x^{m} \in I$ as desired.
\end{proof}

Given any face $F$ of $t\Newt(\ba)$, we will use $\relint ({w \over 1-p^e}+F)$ to denote the relative interior of ${w \over 1-p^e}+F$. By relative interior $\relint(C)$ of a convex set $C$ with positive dimension we mean the interior of $C$ in the affine hull of $C$. The relative interior of a convex set $C$ can be characterized algebraically as follows.

\begin{theorem}[Theorem 3.5 in \cite{BrondstedConvexPolytopes}]
\label{characterization-relint}
For any convex set $C$, the following are equivalent
\begin{enumerate}
\item $x\in \relint(C)$;
\item for any line $A$ in the affine hull of $C$ with $x\in A$, there are points $y_0,y_1\in A\cap C$ such that $x=\delta y_0+(1-\delta)y_1$ for some $\delta\in (0,1)$;
\item for any point $y\in C$ with $y\neq x$, there exists $z\in C$ such that $x=\delta y+(1-\delta)z$ for some $\delta \in (0,1)$.
\end{enumerate}
\end{theorem}

For our purpose, the relative interior of a point is the point itself. For example, there are four faces of the following polyhedron.
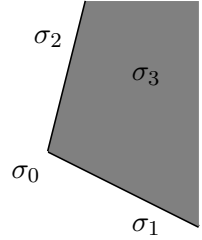
\begin{center}
\begin{tikzpicture}
\begin{scope}[very thick]
\draw (0,0) -- (2,-1) (0,0)--(0.5,2);
\end{scope}
\fill[gray] (0, 0) -- (2, -1) -- (2, 2) -- (0.5, 2) -- cycle;
\draw (-0.3, -0.3) node{$\sigma_0$};
\draw (1.3, -1) node{$\sigma_1$};
\draw (0, 1.5) node{$\sigma_2$};
\draw (1.3, 1) node{$\sigma_3$};
\end{tikzpicture}
\end{center}
The relative interiors of the three positive dimensional faces are pictured as below.
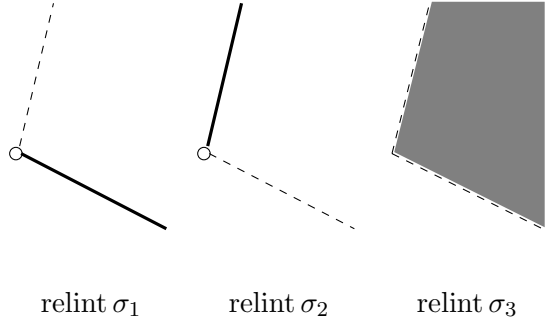
\begin{center}
\begin{tikzpicture}
\begin{scope}[dashed]
\draw  (0.05,0.1) -- (0.5, 2) ;
\end{scope}
\begin{scope}[very thick]
\draw (0.07,0) -- (2, -1);
\end{scope}
\draw (0,0) circle (0.5ex);
\draw (1, -2) node{$\relint \sigma_1$};

\begin{scope}[dashed]
\draw (2.57, 0) -- (4.5, -1) ;
\end{scope}
\begin{scope}[very thick]
\draw (2.55,0.1) -- (3, 2);
\end{scope}
\draw (2.5,0) circle (0.5ex);
\draw (3.5, -2) node{$\relint \sigma_2$};

\fill[gray] (5.03, 0.02) -- (7.03, -0.98) -- (7.03, 2.02) -- (5.53, 2.02) -- cycle;
\begin{scope}[dashed]
\draw (5,0) -- (5.5, 2) (5, 0) -- (7, -1) ;
\end{scope}
\draw (6, -2) node{$\relint \sigma_3$};

\end{tikzpicture}
\end{center}

\begin{proposition}\label{containInterior}
Suppose $I$ is $\mathcal{C}^{\phi,\ba^t}$-fixed and fix a face $F$ of $t\Newt(\ba)$.  If $\{ v | x^v \in I \} \cap ({w \over 1-p^e}+F) \neq \emptyset$, then $\{ v | x^v \in I \}  \supseteq \relint({w \over 1-p^e}+F) \cap S$.
\end{proposition}

\begin{proof} Choose $v_0 \in \{ v | x^v \in I \} \cap ({w \over 1-p^e}+F)$ and $v_0' \in \relint({w \over 1-p^e}+F) \cap S$ such that $v_0 \neq v_0'$. Choose $ \alpha_1 ,\dots, \alpha_k \in {w \over 1-p^e}+F$ so that $v_0' = n_0v_0 + \sum_{i=1}^k n_i \alpha_i$ where $0 < n_i <1$ and $\sum_{i=0}^k n_i =1$.
For $n \gg 0$, $p^{en} n_0 >1$ and
\[ p^{en} v_0' - v_0 = (p^{en}n_0 - 1) v_0 + \sum_{i=1}^k p^{en}n_i \alpha_i \in (p^{en}-1)({w \over 1-p^e}+F) \subseteq {p^{en}-1 \over 1-p^e}w +t(p^{en}-1) \Newt(\ba).\]
By Lemma \ref{propertiesNewt} $x ^{p^{en}v_0'-v_0 + {p^{en}-1 \over p^e-1}w} \in  \overline{ \ba^{t(p^{en}-1)}}$ and hence
$x^{v_0'} = \phi^n ( x^{p^{en} v_0'- v_0 + {p^{en}-1 \over p^e-1}w} \cdot x^{v_0}) \in I$ as desired.
\end{proof}

Now, we are ready to describe the $\mathcal{C}^{\phi,\ba^t}$-fixed ideals. Set $\sF = \{ \text{ faces of } t\Newt(\ba) \}$ (Here, we assume $t \Newt(\ba) \in \sF$). For $F \in \sF$, denote
$I_F :=  \langle  x^v | v \in \relint({w \over 1-p^e}+F')\cap S \text{ for some }F' \supseteq F \rangle.$

\begin{theorem}
\label{cartier-algebra-fixed-ideal}
$I$ is $\mathcal{C}^{\phi,\ba^t}$-fixed if and only if $I = \sum_{F \in \sG} I_F$ for some non-empty $\sG \subseteq \sF$.
\end{theorem}

\begin{proof}
First, suppose $\sum_{n >0} \phi^n \left( \overline{\ba^{t(p^{en}-1)}}\cdot I \right)=I $, where the sum is taken over all integers $n>0$ such that $t(p^{en}-1)$ is an integer.
Let $\sG$ be the subset of $\sF$ consisting of all the faces $F$ satisfying $\langle v | x^v \in I \rangle \cap \relint({w \over 1-p^e}+F) \neq \emptyset$. Then $I \supseteq \sum_{F \in \sG} I_F$ by Proposition \ref{containInterior}. By Proposition \ref{ismonomial}, $x^v \in I$ implies $v \in \relint({w \over 1-p^e}+F) \cap S$ for some $F \in \sG$. So $I \subseteq \sum_{F \in \sG} I_F$.

Conversely, suppose $I = \sum_{F \in \sG} I_F$ for some $\sG \subseteq \sF$.
We first show that  $$\sum_{n >0} \phi^n \left( \overline{\ba^{t(p^{en}-1)}}\cdot I \right)\supseteq I .$$
Let $x^{v} \in I$, then $v \in \relint({w \over 1-p^e}+F) \subseteq
{w \over 1-p^e} + t \Newt(\ba)$ for some face $F$ of $t \Newt(\ba)$.
For any $n>0$ such that $t(p^{en}-1)$ is an integer,
\[ (p^{en}-1)v \in {p^{en}-1 \over 1 - p^e}w +t(p^{en}-1) \Newt(\ba).
\]
Therefore by Lemma \ref{propertiesNewt}, \[ x^{(p^{en}-1)v +{p^{en}-1 \over  p^e-1}w } \in  \overline{\ba^{t(p^{en}-1)}}.
\]
So $ x^v = \phi^n\left( x^{(p^{en}-1)v +{p^{en}-1 \over  p^e-1}w } \cdot x^v \right) \in \phi^n\left(\overline{\ba^{t(p^{en}-1)}} \cdot I \right)
$
as desired.

For the other containment, notice that it suffices to show that
$\sum_{n >0} \phi^n \left( \overline{\ba^{t(p^{en}-1)}}\cdot I_F \right)\subseteq I_F $ for any $F$.
Suppose $x^u \in \sum_{n >0} \phi^n \left( \overline{\ba^{t(p^{en}-1)}}\cdot I_F \right)$ for some $F$, then there exists $n>0$ such that $t(p^{en}-1)$ is an integer and that $$p^{en}u = \alpha+ \beta - {p^{en}-1 \over p^e-1}w$$ for some $\alpha$ and $\beta$ such that $x^{\alpha} \in \overline{\ba^{t(p^{en}-1)}}$ and $x^{\beta} \in I_F$.
In particular, $\alpha \in \Newt(\ba^{t(p^{en}-1)}) = t(p^{en}-1) \Newt(\ba)$ (by Lemma \ref{propertiesNewt}) and $\beta \in {w \over 1-p^e}+\relint \gamma$ for some $\gamma \supseteq F$ (because $x^{\beta} \in I_F$).
Therefore,
\[\begin{aligned}
u- {w \over 1-p^e} &= {\alpha \over p^{en}} + { (\beta - {w \over 1-p^e})\over p^{en}} \\
&\in  {(p^{en}-1) \over p^{en}}  t\Newt(\ba) + {1 \over p^{en}}\relint \gamma \\
&\subseteq t \Newt(\ba).
\end{aligned}\]
In particular, $u - {w \over 1-p^e}  \in \relint \tau$ for some unique face $\tau$ in $\sF$. We claim that $\tau \supseteq \gamma$. This implies that $x^u \in I_{\tau}  \subseteq I_{\gamma} \subseteq I_F$ which finishes the proof.

To see why $\tau \supseteq \gamma$, write $u - {w \over 1-p^e} = {(p^{en}-1) \over p^{en}} a + { 1 \over p^{en}} b$ for some $a \in t \Newt(\ba)$ and $b \in \relint \gamma$. If $\tau = t \Newt(\ba)$, there is nothing to do. Suppose $\tau \neq t \Newt(\ba)$ and let $\sF_{\tau}$ be the set of all maximal faces in $\sF \setminus \{ t \Newt(\ba) \}$ which contain $\tau$. For each face $\theta \in \sF_{\tau}$, fix a non-zero linear functional $f_{\theta}$ such that $f_{\theta}(t\Newt(\ba)) \geq 0$ and $f_{\theta}(\theta)=0$. Then
\[   \tau = \left(\bigcap_{\theta \in \sF_{\tau}} \{ s \in M_{\bR} | f_{\theta}(s) =0 \} \right) \cap (t\Newt \ba).\]
Now, for each $\theta \in \sF_{\tau}$,
\[0 = f_{\theta}(u - {w \over 1-p^e}) = {(p^{en}-1) \over p^{en}} f_{\theta}( a) + { 1 \over p^{en}}f_{\theta}( b).
\]
Since $f_{\theta}(a)$ and $f_{\theta}(b)$ are non negative, we must have $f_{\theta}(a)= f_{\theta}(b) = 0$.
If $\gamma$ is a point, then it is clear that $\gamma=\{b\}$ and hence $f_{\theta}(\gamma)=0$. Assume that $\gamma$ is not a point. Since $b$ is in $\relint(\gamma)$, according to Theorem \ref{characterization-relint}, for each $g \in \gamma$, we can choose $g' \in \gamma$ so that $b=\delta g +(1-\delta)g'$ for some $\delta >0$. So for any $\theta \in \sF_{\tau}$, we have
\[0=f_{\theta}(b)=\delta f_{\theta}(g) +(1-\delta) f_{\theta}(g').\]
This implies $f_{\theta}(g)=0$ for all $g \in \gamma$ and all $\theta \in \sF_{\tau}$. Therefore,
\[   \gamma \subseteq \left(\bigcap_{\theta \in \sF_{\tau}} \{ s \in M_{\bR} | f_{\theta}(s) =0 \} \right) \cap (t\Newt \ba) = \tau.\]
\end{proof}

\begin{corollary}
Given a toric triple $(R,\phi,\ba^t)$ as above, the set of ideals $I$ satisfying the condition that
\[
\sum_{n>0}\phi^n(\overline{\ba^{t(p^{en}-1)}} I) = I
\]
is closed under intersection.
\end{corollary}

\begin{remark}
It also follows from the previous result that in the toric setting for pairs $(R,\phi)$, the set of $\phi$-fixed ideals agrees with the set of $(\phi^2 = \phi \circ F^e_* \phi)$-fixed ideals.  This is also true when $\phi$ is surjective, see for example \cite[Proposition 4.1]{SchwedeFAdjunction}, but it fails in general as the following example shows.
\end{remark}
The following example was clarified to us in conversations with Manuel Blickle, also compare with \cite{DieudonneLieGroupsAndLieHyperII} and \cite[Example 5.28]{BlickleThesis}.

\begin{example}
\label{exInfinitelyManyIdealsEventually}
Suppose that $S = \overline{\bF_5}[x,y,z]$ and $f = x^4 + y^4 + z^4$.  Consider the map $\Phi_S : F_* S \to S$ which sends $x^4y^4z^4 \mapsto 1$ and all other monomials $x^iy^jz^k$ to zero (as long as $0 \leq i,j,k \leq 4 = 5-1$).  Further consider the map $\phi : F_* S \to S$ defined by the rule $\phi(\blank) = \Phi_S(f^4 \cdot \blank)$.  It is easy to verify that $\phi$ induces a map on $R = S/(f)$.

Set $\bm = (x,y,z)$.  One can verify directly that both $\bm$ and $\bm^2 = (x^2, xy, xz, y^2, yz, z^2)$ are $\phi$-fixed (as is $(f)$).  Furthermore, $\phi(x) = 2x$, $\phi(y) = 2y$ and $\phi(z) = 2z$.  Therefore, for any elements $a,b,c \in \overline{\bF_5}$, $\phi(ax + by + cz) = 2 a^{1 \over 5} x + 2 b^{1 \over 5} y + 2 c^{1 \over 5} z$.  Thus the ideal $\bm^2 + (x)$ is also $\phi$-fixed, as is $\bm^2 + (ax + by + cz)$ for any elements $a, b, c \in \bF_5 \subset \overline{\bF_5}$ (this is still a finite set).

However, if one considers $\phi^2 = \phi \circ F_* \phi : F^2_* S \to S$, then one has $\phi^2(ax + by + cz) = 4 a^{1 \over 25} x + 4 b^{1 \over 25} y + 4 c^{1 \over 25} z$ so that $\bm^2 + (ax + by + cz)$ is $\phi^2$-fixed for any elements $a,b,c \in \bF_{5^2} \subseteq \overline{\bF_5}$.  Thus we have found ideals that are $\phi^2$-fixed but not $\phi$-fixed.  Even more, continuing in this way, one obtains that the set of $\phi^n$-fixed ideals can become arbitrarily large as $n$ increases.

As we saw, toric varieties do not exhibit this phenomena.
\end{example}

\begin{remark}
One can also ask what happens if $t \geq 0$, the coefficient of $\ba$, has a $p$ in its denominator.  Unfortunately, in that case the non-$F$-pure ideal (the largest compatible ideal) seems to typically coincide with the test ideal (the smallest non-$F$-pure ideal) and so the lattice of compatible ideals is uninteresting.  For further discussion, see \cite{FujinoSchwedeTakagiSupplements}.
\end{remark}

\begin{example} As suggested by the referee, we include an example to complement the discussion on page 2.

Consider $k[S]=k[x,xy,xy^2,xy^3]$ and as before the $p^{-e}$-linear map $\phi(\blank)=\phi_c(x^{-w}\cdot \blank)$ fixing $\ba = R$.
We compute the fixed ideals for various values of $w$ and $p^e$, some of these are pictured below.
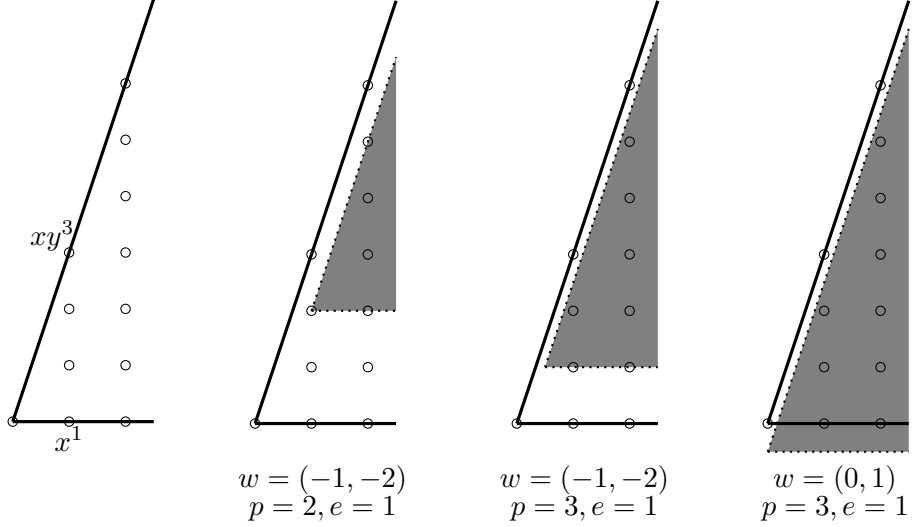
\begin{center}
\begin{tikzpicture}[scale = 0.75]
\draw (0,0) circle (0.5ex);
\draw (1,0) circle (0.5ex);
\draw (1,1) circle (0.5ex);
\draw (1,2) circle (0.5ex);
\draw (1,3) circle (0.5ex);
\draw (2,0) circle (0.5ex);
\draw (2,1) circle (0.5ex);
\draw (2,2) circle (0.5ex);
\draw (2,3) circle (0.5ex);
\draw (2,4) circle (0.5ex);
\draw (2,5) circle (0.5ex);
\draw (2,6) circle (0.5ex);
\draw (1, -0.3) node{$x^1$};
\draw (0.7, 3.3) node{$xy^3$};
\begin{scope}[very thick]
\draw (0,0)--(2.5,7.5);
\draw (0,0)--(2.5,0);
\draw (1, -1.75) node{$\text{ }$};
\end{scope}
\end{tikzpicture}  
\hskip 24pt  
\begin{tikzpicture}[scale = 0.75]
\fill[gray] (1,2) -- (2.5, 6.5) -- (2.5,2) -- cycle;
\draw (0,0) circle (0.5ex);
\draw (1,0) circle (0.5ex);
\draw (1,1) circle (0.5ex);
\draw (1,2) circle (0.5ex);
\draw (1,3) circle (0.5ex);
\draw (2,0) circle (0.5ex);
\draw (2,1) circle (0.5ex);
\draw (2,2) circle (0.5ex);
\draw (2,3) circle (0.5ex);
\draw (2,4) circle (0.5ex);
\draw (2,5) circle (0.5ex);
\draw (2,6) circle (0.5ex);
\begin{scope}[very thick]
\draw (0,0)--(2.5,7.5);
\draw (0,0)--(2.5,0);
\end{scope}
\begin{scope}[thick, dotted]
\draw(1, 2) -- (2.5, 6.5);
\draw(1, 2) -- (2.5, 2);
\end{scope}

\draw (1.2, -1) node{$w = (-1,-2)$};
\draw (1.2, -1.5) node{$p = 2, e = 1$};
\end{tikzpicture}
\hskip 24pt  
\begin{tikzpicture}[scale = 0.75]
\fill[gray] (0.5,1) -- (2.5, 7) -- (2.5,1) -- cycle;
\draw (0,0) circle (0.5ex);
\draw (1,0) circle (0.5ex);
\draw (1,1) circle (0.5ex);
\draw (1,2) circle (0.5ex);
\draw (1,3) circle (0.5ex);
\draw (2,0) circle (0.5ex);
\draw (2,1) circle (0.5ex);
\draw (2,2) circle (0.5ex);
\draw (2,3) circle (0.5ex);
\draw (2,4) circle (0.5ex);
\draw (2,5) circle (0.5ex);
\draw (2,6) circle (0.5ex);
\begin{scope}[very thick]
\draw (0,0)--(2.5,7.5);
\draw (0,0)--(2.5,0);
\end{scope}
\begin{scope}[thick, dotted]
\draw(0.5, 1) -- (2.5, 7);
\draw(0.5, 1) -- (2.5, 1);
\end{scope}

\draw (1.2, -1) node{$w = (-1,-2)$};
\draw (1.2, -1.5) node{$p = 3, e = 1$};
\end{tikzpicture}
\hskip 24pt  
\begin{tikzpicture}[scale = 0.75]
\fill[gray] (0,-0.5) -- (2.5, 7) -- (2.5,-0.5) -- cycle;
\draw (0,0) circle (0.5ex);
\draw (1,0) circle (0.5ex);
\draw (1,1) circle (0.5ex);
\draw (1,2) circle (0.5ex);
\draw (1,3) circle (0.5ex);
\draw (2,0) circle (0.5ex);
\draw (2,1) circle (0.5ex);
\draw (2,2) circle (0.5ex);
\draw (2,3) circle (0.5ex);
\draw (2,4) circle (0.5ex);
\draw (2,5) circle (0.5ex);
\draw (2,6) circle (0.5ex);
\begin{scope}[very thick]
\draw (0,0)--(2.5,7.5);
\draw (0,0)--(2.5,0);
\end{scope}
\begin{scope}[thick, dotted]
\draw(0, -0.5) -- (2.5, 7);
\draw(0, -0.5) -- (2.5, -0.5);
\end{scope}

\draw (1.2, -1) node{$w = (0,1)$};
\draw (1.2, -1.5) node{$p = 3, e = 1$};
\end{tikzpicture}

\end{center}
\begin{itemize}
\item $w=(-1,-2)$: \\
When $p^e=2$, the $\phi$-fixed ideals are \\
$0, \langle x^2y^3,x^2y^4\rangle, \langle x^2y^2,x^2y^3,x^2y^4\rangle,
\langle x^2y^3,x^2y^4,x^2y^5\rangle,\langle x^2y^2,x^2y^3,x^2y^4,x^2y^5\rangle, \langle xy^2 \rangle.$\\
When $p^e=3$, the $\phi$-fixed ideals are $0, \langle xy^2 \rangle, \langle xy,xy^2\rangle$.\\
When $p^e \geq 4$, the only $\phi$-fixed ideals are $0$ and $\langle xy,xy^2 \rangle$.
\item $w=(1,2)$: The $\phi$-fixed ideals are $0,k[S]$.
\item $w=(1,0)$: The $\phi$-fixed ideals are $0,\langle xy,xy^2,xy^3\rangle,k[S]$.
\item $w=(0,1)$: \\
When $p^e=2$, the $\phi$-fixed ideals are $0,\langle x,xy \rangle,\langle x,xy,xy^2 \rangle$. \\
When $p^e \geq 3$, the $\phi$-fixed ideals are $0,\langle x,xy,xy^2 \rangle$.
\end{itemize}
\end{example}

\section{Intermediate adjoint ideals for triples}
\label{secIntAdjointIdealsTrio}

Suppose that $X$ is a normal variety and that $Z$ is a closed subset of $X$ such that $X \setminus Z$ is dense in $X$ and that $\Delta$ is a (often effective) $\bQ$-divisor such that $K_X + \Delta$ is $\bQ$-Cartier. Let $\ba$ is an ideal sheaf on $X$. Assume that the triple $(X, \Delta, \ba^t)$ admits a log resolution.

We consider log resolutions $\pi :  X' \to X$ of $X, Z, \Delta,\ba^t$ with $\ba\cdot\O_{X'}=\O_{X'}(-G)$.  In this context, we can write $K_{ X'} - \pi^* (K_X + \Delta)-tG = \sum_{i} a_i E_i$ (as is standard, for any proper birational map $\pi : \tld{X} \to X$, we assume that $\pi_* K_{\tld X} = K_X$).  For each such log resolution $\pi$, we have the following set of (possibly fractional) ideals:
\begin{align*}
 \sI_{Z}^{\pi}(X, \Delta,\ba^t)
= & \{ \pi_* \O_{ X'}(\lceil K_{X'} - \pi^*(K_X + \Delta) -tG+ E\rceil) | \\
& E \text{ is a reduced divisor satisfying $\pi(E) \subseteq Z$}  \\
& \text{also such that each component $E_i$ of $E$ has associated $a_i \in \mathbb{Z}$.}  \}
\end{align*}
\begin{definition}
We define
\[
\sI_{Z}(X, \Delta,\ba^t) := \bigcup_{\pi} \sI_Z^{\pi}(X, \Delta,\ba^t).
\]
We call this set the set of \emph{intermediate adjoint ideals with respect to $Z$}.
\end{definition}
We also give an alternative characterization of $\sI_Z^{\pi}(X, \Delta,\ba^t)$.
\begin{lemma}
\label{lemAltCharInterAdj}
The set of ideals $\sI_Z^{\pi}(X, \Delta,\ba^t)$ is equal to
\[
\begin{split}
 \{ \pi_* \O_{X'}(\lceil K_{X'} - \pi^*(K_X + \Delta) - tG+\varepsilon E\rceil) | E \text{ an effective divisor satisfying $\pi(E) \subseteq Z$} \\
\text{$\varepsilon$ satisfying $1 \gg \varepsilon > 0$.} \}
\end{split}
\]

\end{lemma}
\begin{proof}
It is obvious.
\end{proof}

We also recall several related definitions; the \emph{multiplier ideal} and the \emph{maximal non-LC ideal}.

\begin{definition}\cite{LazarsfeldPositivity2,FujinoSchwedeTakagiSupplements,KawamataSubadjunction2}
Suppose that $X$ is a normal variety over a field of characteristic zero, that $\Delta$ is an effective $\bQ$-divisor such that $K_X + \Delta$ is $\bQ$-Cartier, $\ba \subseteq \O_X$ is an ideal sheaf and $t \geq 0$ is a real number.  Further suppose that $\pi : X' \to X$ is a log resolution of $(X, \Delta, \ba)$ where we set $\ba \cdot \O_{X'} = \O_{X'}(-G)$.
\begin{itemize}
\item{} The \emph{multiplier ideal} $\mJ(X, \Delta, \ba^t)$ is defined to be \[
\pi_* \O_{X'}(\lceil K_{X'} - \pi^* (K_X + \Delta) - tG \rceil)\]
This ideal is independent of the choice of resolution $\pi$.  If $\mJ(X, \Delta, \ba^t) = \O_X$, then $(X, \Delta, \ba^t)$ is said to have \emph{Kawamata log terminal (or klt)} singularities.
\item{}  If $E$ the reduced divisor on $X'$ with support equal to $\Supp(G) \cup \Supp(\pi^{-1}_* \Delta) \cup \exc(\pi)$, then the \emph{maximal non-LC-ideal} $\mJ'(X, \Delta, \ba^t)$ is defined to be
    \[
    \pi_* \O_{X'}(\lceil K_{X'} - \pi^* (K_X + \Delta) - tG  + \varepsilon E\rceil )
    \]
    where $\varepsilon > 0$ is arbitrarily small.  This ideal is independent of the choice of resolution $\pi$ assuming that $1 \gg \varepsilon > 0$.  If $\mJ'(X, \Delta, \ba^t) = \O_X$, then $(X, \Delta, \ba^t)$ is said to have \emph{log canonical (or lc)} singularities.
\item{}  If $W \subseteq X$ is an irreducible closed subvariety, then $W$ is said to be a \emph{log canonical center} (or \emph{lc center}) of $(X, \Delta, \ba^t)$ if $(X, \Delta, \ba^t)$ is lc at the generic point of $W$ but not klt at the generic point of $W$.
\end{itemize}
\end{definition}

We prove a number of basic results about the sets $\sJ^{\pi}_Z(X, \Delta,\ba^t)$ and $\sJ_Z(X, \Delta,\ba^t)$.

\begin{lemma}
\label{LemPropertiesOfInterAdj}
Suppose that $X$, $Z$, $\Delta$, $\ba$ and $\pi : X' \to X$ are as above.  Then the following hold:
\begin{itemize}
\item[(i)]  If $Y \subseteq Z$ is closed and $\pi$ is also a log resolution of $Y$, then $\sI_Y^{\pi}(X, \Delta,\ba^t) \subseteq \sI_Z^{\pi}(X, \Delta ,\ba^t)$.
\item[(ii)]  For any Cartier divisor $L$ such that $\pi$ is a log resolution of $\Delta$ and $\Delta + L$, we have that
\[
\{ \O_X(-L) \tensor I | I \in \sI_Z^{\pi}(X, \Delta,\ba^t) \} = \sI_Z^{\pi}(X, \Delta+L,\ba^t)	
\]
\item[(iii)]  $\sI_Z^{\pi}(X, \Delta,\ba^t) = \{ \pi_* J | J \in \sI_{\pi^{-1}(Z)}^{\id}(X', \pi^*(K_X + \Delta) - K_{X'},(\ba\cdot\O_{X'})^t ) \}$.
\item[(iv)]   Suppose that $\pi' : X'' \to X'$ is a proper birational map such that $\pi \circ \pi' : X'' \to X$ is also a log resolution.  Then $\sI_Z^{\pi \circ \pi'}(X, \Delta,\ba^t) \supseteq \sI_Z^{\pi}(X, \Delta,\ba^t)$.
\item[(v)]  For any proper birational map $\tau : Y \to X$ with $Y$ normal, then
\[
\sI_Z(X, \Delta,\ba^t) = \{ \pi_* I | I \in \sI_{\tau^{-1}(Z)}(Y, -K_Y + \tau^*(K_X + \Delta),(\ba\cdot \O_Y)^t) \}.
\]
\item[(vi)]  If $\Delta$ is effective, $(X, \Delta,\ba^t)$ is log canonical and $Z$ contains the non-Kawamata-log terminal locus of $(X, \Delta,\ba^t)$ and also satisfies $Z \subseteq \Sing(X) \cup \Supp(\Delta)\cup \Supp(G)$, then $\sI_Z(X, \Delta,\ba^t)$ is the set of ideals defining all unions of log canonical centers of $(X, \Delta,\ba^t)$.
\item[(vii)]  If $Z = \Sing X \cup \Supp \Delta \cup V(\ba)$ and $\Delta \geq 0$, then the unique smallest element of $\sI_Z(X, \Delta,\ba^t)$ is the multiplier ideal $\mJ(X, \Delta,\ba^t)$ and the
unique largest element is the maximal non-LC-ideal $\mJ'(X, \Delta,\ba^t)$ of \cite{FujinoSchwedeTakagiSupplements}.
\item[(viii)]  The set $\sI_Z(X, \Delta,\ba^t)$ is finite.
\item[(ix)]  The set $\sI_Z(X, \Delta,\ba^t)$ is closed under intersection.
\end{itemize}
\end{lemma}
\begin{proof}
(i) is obvious.  (ii) is a direct consequence of the projection formula.  (iii) is clear from the definition.  To prove (iv), notice that for any divisor $E$ on $X'$ such that $\Supp E \subseteq \pi^{-1}(Z)$, we have that
\[
\O_{X'}(\lceil K_{X'} - \pi^*(K_X + \Delta) -tG+ \varepsilon E \rceil) = \pi'_* \O_{X''}(\lceil K_{X''} - \pi'^* \pi^*(K_X + \Delta) -t\pi'^{*}G+ \varepsilon \pi'^*  E) \rceil)
\]
by \cite[Lemma 9.2.19]{LazarsfeldPositivity2}.  The result then immediately follows from Lemma \ref{lemAltCharInterAdj}.  For (v), simply notice that by (iv), we may restrict ourselves to $\pi : X' \to X$ which factor through $\tau$.    (vi) is obvious from the definition.   (vii) is the definition of the multiplier ideal and maximal non-LC-ideal respectively.

Now we prove (viii).  By (v), we may assume that $X$ is smooth and that the support of $\ba$ is a divisor with simple normal crossings with $-K_Y + \tau^*(K_X + \Delta)$.  Thus we may incorporate $\ba$ into the divisor term.  On the other hand using (ii), we may assume that the coefficients of this divisor are between $0$ and $1$.  Thus we have reduced our situation to simply proving the finiteness of the set of log canonical centers, which is well known \cf \cite{AmbroBasicPropertiesOfLCCenters}.
For (ix), suppose that $I$ and $J$ are contained in $\sI_Z(X, \Delta,\ba^t)$.  By (iv), we may realize both ideals as push forward from a single log resolution, say $I = \pi_* \O_{X'}(\lceil K_{X'} - \pi^*(K_X + \Delta) -tG+ \varepsilon E \rceil)$ and $J = \pi_* \O_{X'}(\lceil K_{X'} - \pi^*(K_X + \Delta) -tG+ \varepsilon F \rceil)$ for some sufficiently small positive $\varepsilon$ where $E$ and $F$ are reduced divisors with image contained in $Z$. Let $E\wedge F$ denote the sum of the common components of $E$ and $F$, \ie, we can write $E=E\wedge F+E'$ and $F=E\wedge F+F'$ where $\Supp(E\wedge F)$, $\Supp(E')$ and $\Supp(F')$ are componentwise disjoint. Note that $E\wedge F$ is also a reduced divisor with $\pi(E\wedge F)\subseteq Z$. We will show that $I \cap J = \pi_* \O_{X'}(\lceil K_{X'} - \pi^*(K_X + \Delta)-tG + \varepsilon (E \wedge F) \rceil )$.

For the $\subseteq$ direction, suppose that $g \in I \cap J$, then $\Div_{X'}(g) + \lceil K_{X'} - \pi^*(K_X + \Delta)-tG + \varepsilon E \rceil \geq 0$ and $\Div_{X'}(g) + \lceil K_{X'} - \pi^*(K_X + \Delta)-tG + \varepsilon F \rceil \geq 0$.  Since $\Supp(E\wedge F)$, $\Supp(E')$ and $\Supp(F')$ are componentwise disjoint, it then immediately follows that
\[\Div_{X'}(g) + \lceil K_{X'} - \pi^*(K_X + \Delta)-tG + \varepsilon (E \wedge F)\rceil  \geq 0.\]  Conversely, suppose that $g \in \pi_* \O_{X'}(\lceil K_{X'} - \pi^*(K_X + \Delta) -tG+ \varepsilon (E \wedge F) \rceil )$.  Thus $\Div_{X'}(g) + \lceil K_{X'} - \pi^*(K_X + \Delta)-tG + \varepsilon (E \wedge F)\rceil \geq 0$ which immediately implies that $g \in I$ and $g \in J$ as well since $E \wedge F \leq E, F$.
\end{proof}

We now describe a method for producing intermediate adjoint ideals which will be crucial in the next section.

\begin{proposition}
\label{PropHowToConstructInterAdj}
Suppose that $(X = \Spec R, \Delta, \ba^t)$ is a triple, that $Z = \Sing X \cup \Supp \Delta \cup V(\ba)$ and $\Delta \geq 0$.  Then for any $0 \neq \bb \subseteq R$, $\mJ'(X, \Delta, \ba^t\bb^{\delta}) \in \sI_Z(X, \Delta,\ba^t)$ for sufficiently small $\delta > 0$.
\end{proposition}
\begin{proof}
Choose a log resolution $\pi : X' \to X$ of $(X, \Delta, \ba^t\bb^{\delta})$ such that $\ba \cdot \O_{X'} = \O_{X'}(-G)$ and $\bb \cdot \O_{X'}=\O_{X'}(-H)$.  Set
\[E = \Supp(\pi^*(K_X+\Delta) + tG+ \delta H - K_{X'})^{\geq 1}.\]
Note that $E=\Supp(\pi^*(K_X+\Delta) + tG -K_{X'})^{\geq 1}$ for sufficiently small $\delta>0$, and hence $\pi(E)\subseteq Z$. By definition,
\[
\mJ'(X, \Delta, \ba^t\bb^{\delta} )  = \pi_* \O_{X'}(\lceil K_{X'} - \pi^*(K_X + \Delta ) - tG - \delta H + \varepsilon E \rceil)
\]
where $0 < \varepsilon \ll \delta$.  Set $F$ to be the reduced divisor made up of all components of $E$ that do not appear in $\Supp(H)$ (it might be that $F = 0$).  Then we claim that $\pi(F)\subseteq Z$ (which is obvious) and also that
\begin{equation}
\label{eqnEqualityOfRounds}
\lceil K_{X'} - \pi^*(K_X + \Delta) - tG - \delta H + \varepsilon E \rceil = \lceil K_{X'} - \pi^*(K_X + \Delta) -tG + \varepsilon F \rceil
\end{equation}
again for $\delta \gg \varepsilon > 0$ sufficiently small.  This will complete the proof.  To see this second claim, consider $a_i$, the $E_i$-coefficient of $K_{X'} - \pi^*(K_X + \Delta) - tG$, where $E_i$ is a component of $E$.  There are two cases, if $E_i$ is a component of $H$, then before rounding up, the $E_i$-coefficient of the left side of Equation \ref{eqnEqualityOfRounds} is $a_i - \delta + \varepsilon$ while the $E_i$-coefficient of the right side is $a_i$, these have the same round-up since $1 \gg \delta \gg \varepsilon > 0$.  On the other hand, if $E_i$ does not appear in $H$, then along $E_i$, Equation \ref{eqnEqualityOfRounds} is already an equality and there is nothing to show.
\end{proof}

\section{Intermediate adjoint ideals on toric varieties}
\label{secIntAdjointOnToric}
In this section we will work only with toric varieties.
Throughout this section, we fix $M$ to be a lattice and $\sigma$ to be a rational convex polyhedral cone in $M_{\bR}$, set $S = \sigma \cap M$. Let $X$ be the affine toric variety $\Spec k[S]$ and suppose that $\Delta \geq 0$ is a torus-invariant $\bQ$-divisor on $X$ such that $K_X + \Delta$ is $\bQ$-Cartier. Let $\ba$ be a toric ideal of $\Spec k[S]$ and $t$ be a nonnegative real number.
Set $Z = \Supp \Delta \cup \Sing X \cup V(\ba)$, in this section we describe the set of ideals $\sI_Z(X, \Delta, \ba^t)$.

\begin{lemma}
The ideals of $\sI_Z(X, \Delta, \ba^t)$ are torus invariant.
\end{lemma}
\begin{proof}
First suppose that $\pi : X' \to X$ is a toric log resolution.  By Lemma \ref{LemPropertiesOfInterAdj}(v), it is enough to show that
$\sI_{\pi^{-1} Z}(X', -K_{X'} + \pi^*(K_X + \Delta), \ba^t\cdot \O_{X'})$ is a collection of torus invariant (fractional) ideals since the push forward of a torus invariant fractional ideal is still torus invariant.
However, using the trick of \ref{LemPropertiesOfInterAdj}(ii), it is easy to see that, up to twisting by a line bundle on, $\sI_{\pi^{-1} Z}(X', -K_{X'} + \pi^*(K_X + \Delta), \ba^t\cdot \O_{X'})$ is just a finite collection of log canonical centers of a torus invariant triple on $X'$ (again, by blowing up $\ba$, one may change the question into the study of the log canonical centers of a pair).  This completes the proof.
\end{proof}

\begin{remark}
In the toric setting and any characteristic, we have a sufficiently good theory of resolution of singularities to generalize the notions from Section \ref{secIntAdjointIdealsTrio} (alternately, one could define similar notions by considering all proper birational maps instead of log resolutions).  Therefore, since in this section we work only with toric varieties, we can now work in any characteristic (either characteristic $p > 0$ or characteristic zero).
\end{remark}

We now describe the non-LC ideal sheaf $\mJ'(X, \Delta, \ba^t)$ in the toric language.
First choose $l \in \bN$ such that $l(K_X + \Delta) = \Div_X(x^{m})$ for some $m \in M$.

\begin{proposition}\label{toricnonLC}
\[\mJ'(X,\Delta,\ba^t\bb^s) = \langle  x^v \mid v - {m \over l} \in t\Newt(\ba) + s\Newt(\bb)\rangle\]
where $\Newt(\ba)$ is the Newton polygon of $\ba$.
In particular, the ideal $\mJ'(X, \Delta,\ba^t)$ is generated by the monomials $x^v$ such that $v \in \sigma$ and $v - {m \over l}\in t\Newt(\ba)$.
\end{proposition}
\begin{proof}
Choose a log resolution of $(X, \Delta, \ba, \bb)$, $\pi : X' \to X$. Set $\ba \cdot \O_{X'} = \O_{X'}(-G)$ and $\bb \cdot \O_{X'} = \O_{X'}(-H)$.
Write $K_{X'} - \pi^*(K_X + \Delta) - tG -sH= \sum a_i E_i$ and set $E$ to be the reduced divisor whose components are made up of $E_i$ such that the corresponding $a_i \leq -1$.
It immediately follows that $x^v \in \mJ'(X,\Delta,\ba^t\bb^s)$ if and only if
\[ \Div_{X'}(x^v) + \lceil K_{X'} - \pi^*(K_X+ \Delta) -t G -sH + \varepsilon E  \rceil\geq 0.
\]

We claim that this is true if and only if $\Div_{X'}(x^v)-\pi^*(K_X+\Delta) -t G -sH\geq 0$ and we reason as follows.

Let $D_i$ be a torus invariant divisor on $X$ and let $c_i$ be the coefficient of $\Div_{X'}(x^v)$ on $D_i$.
Further set $b_i$ to be the coefficient of $-\pi^*(K_X+\Delta) -t G -sH$ along $D_i$.
It follows that the coefficient of $\lceil K_{X'} - \pi^*(K_X+ \Delta) -t G  -sH + \varepsilon E  \rceil$ along $D_i$ is $\lceil -(1 - \varepsilon) + b_i \rceil$.
Thus our claim is simply that $c_i \geq -b_i$ if and only if $c_i \geq \lfloor - b_i + (1 - \varepsilon)  \rfloor$ for $1 \gg \varepsilon > 0$. But this is obvious.

Notice that the integral closure $\overline{a} = \pi_*\O_{X'}(-G) = \langle x^v \mid v \in \Newt(\ba) \rangle$ and $\overline{b} = \pi_*\O_{X'}(-H) = \langle x^v \mid v \in \Newt(\bb) \rangle$. If $t >0$ is a rational number, we see that $x^v \in \mJ'(X,\Delta, \ba^t\bb^s)$ if and only if $v - {m \over l} \in t\Newt(\ba) + s\Newt(\bb)$ as stated.
The general case is achieved by taking a limit.
\end{proof}

We now transition to characteristic $p > 0$. Let $X$ be a normal toric variety and $\Delta$ be a toric-invariant divisor such that $\Div_X(x^w) = (1-p^e)(K_X + \Delta)$ for some positive integer $l$ and some element $w\in M$. Denote the $p^{-e}$-linear map corresponding to $\Delta$ by $\phi_{\Delta}$, or simply $\phi$ when $\Delta$ is clear. Choose a rational number $t \geq 0$ without $p$ in its denominator. We will show that $\sI_Z(X, \Delta,\ba^t)$ coincides with the ideals fixed by the Cartier algebra $\mathcal{C}^{\phi,\ba^t}$ generated by $\phi$ and $\ba^t$, \ie
\[I\in \sI_Z(X, \Delta,\ba^t)\]
if and only if
\[\sum_{n>0}\phi^n(\overline{\ba^{t(p^{en}-1)}}\cdot I)=I,\]
where the sum is taken over all $n$ such that $t(p^{en}-1)$ is an integer.

Set $P=t\Newt(\ba)$ and $\sF=\{ {\rm faces\ of\ }P\}$. For $F\in \sF$, denote
\[I_F:=\langle x^v | v \in\ \relint(\frac{w}{1-p^e}+F')\cap S \ {\rm for\ some\ }F'\supseteq F\rangle.\]

\begin{theorem}
\label{fixed-ideals-are-adjoint-ideals}
Suppose that $(X, \Delta,\ba^t)$ is as above and that $\phi: F^e_*  \O_X \to \O_X$ corresponds to $\Delta$.
Then every $\mathcal{C}^{\phi, \ba^t}$-fixed ideal appears in the set $\sI_Z(X, \Delta,\ba^t)$.
\end{theorem}
\begin{proof}
We first show that $I_{F} \in \sI_Z(X, \Delta,\ba^t)$ for each $F \in \sF$. Take $a \in \relint(F)$, then $n'a \in \relint(F') \cap M$ for some positive integer $n'$ and some $F' \supseteq F$.

By Proposition \ref{PropHowToConstructInterAdj}, $\mJ'(X, \Delta+ {1 \over n} \Div_X(x^{n'a}),\ba^t) \in \sI_Z(X, \Delta,\ba^t)$ for sufficiently large $n \in \N$. Notice that \[ n(p^e-1)(K_X + \Delta + {1 \over n} \Div_X(x^{n'a})) = -n \Div_X(x^w) + (p^e-1) \Div_X(x^{n'a}) = \Div_X( x^{-nw+(p^e-1)n'a}).\]
By Proposition \ref{toricnonLC}, $\mJ'(X, \Delta+ {1 \over n} \Div_X(x^{n'a}),\ba^t)$ is generated by the monomials $x^v$ such that
\[
v +\frac{1}{p^e-1}w  - {n' \over n} a \in\ t\Newt(\ba).
\]
Therefore, $\mJ'(X, \Delta+ {1 \over n} \Div_X(x^{n'a}),\ba^t)$ coincides with the ideal $I_{F}$ for sufficiently large $n$.

To complete the proof, we must show that for any non-empty subset $\sG$ of $\sF$, the ideal \[ \sum_{F \in \sG} I_{F} \in \sI_Z(X,\Delta,\ba^t).\]
For each $F \in \sG$, pick a point $b_{F}$ in the relative interior of $F$.
Also pick a positive integer $n_F$ and a face $F'$ such that $F' \supseteq F$ and $n_F b_F \in \relint (F') \cap M$.

Consider the ideal
$\bb = \langle x^{n_F b_{F}} \mid F \in \sG \rangle$ with integral closure $\bar{\bb}$.
We claim that for $0 < \delta \ll 1$,
\[
\sum_{F \in \sG} I_{F} = \mJ'(X,\Delta,\ba^t\bb^{\delta}) = \mJ'(X,\Delta,\bar{\ba}^t\bar{\bb}^{\delta}).
\]
Choose a toric log resolution $\pi : X' \rightarrow X$ for $X,\Delta,\ba,\bb$ so that $\ba \cdot \O_{X'} = \O_{X'}(-G)$ and $\bb\cdot \O_{X'}=\O_{X'}(-H)$ for some Cartier divisors $G$ and $H$.
Since for each $F \in \sG$, we have $\Div_{X'}(x^{n_Fb_{F}}) \geq H$, it follows that
\[ \begin{aligned}
I_F &= \mJ'(X, \Delta + \delta \Div_X(x^{n_Fb_{F}}),\ba^t)  \\
            &= \pi_* \O_{X'}(\lceil  K_{X'} - \pi^*(K_X +\Delta)-tG -\delta \Div_{X'}(x^{n_Fb_{F}}) + \varepsilon E \rceil) \\
            &\subseteq \pi_* \O_{X'}(\lceil  K_{X'} - \pi^*(K_X +\Delta) - tG -\delta H + \varepsilon E \rceil) \\
	    &=\mJ'(X,\Delta,\ba^t\bb^{\delta}).
\end{aligned}
\]
Thus we have proven one containment.

Given a face $\gamma$ of $t\Newt(\ba)$, set $T_{\gamma} = \{ x^v | v \in \relint(\frac{w}{1-p^e} + \gamma) \cap S \}$.  By Proposition \ref{toricnonLC}, $\mJ'(X,\Delta,\ba^t\bb^{\delta}) \subseteq \langle x^v | v \in {w \over 1-p^e}+t \Newt(\ba) \rangle$. Therefore, in order to show that
$\mJ'(X,\Delta,\ba^t\bb^{\delta}) \subseteq \sum_{F \in \sG} I_F$,
 it suffices to show that, if $\mJ'(X,\Delta,\ba^t\bb^{\delta})
\cap T_{\gamma} \neq \emptyset$   for some face $\gamma$ of $t\Newt(\ba)$, then $\gamma$ contains some face $F \in \sG$ (\ie $T_{\gamma} \subseteq I_F$).
Let $x^v$ be a nonzero element in $\mJ'(X,\Delta,\ba^t\bb^{\delta}) \cap T_{\gamma} $, then
$v+ \frac{w}{p^e-1} \in \relint \gamma$ and by Proposition \ref{toricnonLC},
$v+ \frac{w}{p^e-1} \in t\Newt(\ba) + \delta\Newt(\bb)$.
Suppose $\gamma$ does not contain any faces in $\sG$.
Then $\gamma \neq t\Newt(\ba)$. Let $\sF_{\gamma}$ be the subset of $\sF$ consisting of all facets (maximal faces)
$\tau \in \sF \setminus \{ t\Newt(\ba) \}$ satisfying $\tau \supseteq \gamma$. For each $\tau \in \sF_{\gamma}$, fix  a non-zero linear functional
$f_{\tau}$ such that $f_{\tau}(\tau)=0$ and $f_{\tau}(t\Newt(\ba)) \geq 0$.
Then \[   \gamma = \left(\bigcap_{\tau \in \sF_{\gamma}} \{ u \in M_{\bR} | f_{\tau}(u) =0 \} \right) \cap (t\Newt \ba).\]
  In particular, $f_{\tau}(v+ {w \over p^e-1} )= 0 $ for all $\tau \in \sF_{\gamma}$.\\
On the other hand, recall that (Lemma \ref{propertiesNewt}) for any monomial ideal $I=\langle x^{g} |g \in \Gamma \rangle$ in $k[S]$
\[ \Newt(I)= \Conv{\Gamma} + \sigma \] where $\Conv{\Gamma}$ is the convex hull of $\Gamma$. Note also that $\xi \sigma = \sigma$ for any $\xi >0$.
Now, suppose $\ba = \langle x^g | g \in \Gamma_{\ba} \rangle$. Then
\[\begin{aligned}
t\Newt(\ba) + \delta \Newt{\bb}
&=t \Conv( \Gamma_{\ba})+ t \sigma
+ \delta \Conv \{ n_Fb_F | F \in \sG \} + \delta \sigma \\
&= t\Newt(\ba) + \delta \Conv \{ n_Fb_F | F \in \sG \}.
\end{aligned}
\]
Write $v+{w \over p^e-1} = \alpha + \delta \sum_{F \in \sG} m_F(n_Fb_F)$ where $\alpha \in t \Newt(\ba)$, $m_F \geq 0$, and $\sum m_F = 1$.
Pick an $F_0$ so that $m_{F_0} \neq 0$. Since $\gamma$ does not contain $F_0$, there exists $\tau_0 \in \sF_{\gamma}$ such that $f_{\tau_0}(b_{F_0}) >0$. In particular,
$f_{\tau_0}(v+{w \over p^e-1}) = f_{\tau_0}(\alpha) + \delta \sum_{F \in \sG} m_Fn_Ff_{\tau_0}(b_F) >0$. This is a contradiction.
\end{proof}

Our next goal is to prove that every element of $\sI_Z(X, \Delta,\ba^t)$ is $\mathcal{C}^{\phi, \ba^t}$-fixed.

\begin{theorem}
\label{adjoint-ideals-are-fixed}
Suppose that $(X, \Delta,\ba^t)$ is as above and that $\phi : F^e_*  \O_X \to \O_X$ corresponds to $\Delta$ as in Section \ref{secTestModulesonToricVarieties}.  Suppose further that $t\geq 0$ is a rational number such that $t(p^e - 1)$ is an integer.  If $J \in \sI_Z(X, \Delta,\ba^t)$, then $J$ is $\mathcal{C}^{\phi, \ba^t}$-fixed.
\end{theorem}

\renewcommand{\chi}{x}

\begin{proof}
Let $\mathcal{J}(E)$ denote $\pi_*\O_{X'}(\lceil K_{X'}-\pi^*(K_X+\Delta)- tG +\varepsilon E\rceil)$. First we show that $\mathcal{J}(E)$ is $\mathcal{C}^{\phi, \ba^t}$-stable, \ie $\sum_{n>0} \phi^n(\overline{\ba^{t(p^{en}-1)}}\mathcal{J}(E))\subseteq \mathcal{J}(E)$. It suffices to show that $\phi^n(x^u\cdot \mathcal{J}(E))\subseteq \mathcal{J}(E)$ for all $x^u\in \overline{\ba^{t(p^{en}-1)}}$.

Write $\frac{1}{1-p^e}\Div_{X'}(\chi^w)=\pi^*(K_X+\Delta)=\sum_ia_iD_i$, where $\{D_i\}$ is the set of toric prime Weil divisors on $X'$ and let $v_i$ be the first lattice point on the ray associated with $D_i$. Hence

\begin{equation}
\label{coefficients-w}
\langle w,v_i\rangle =(1-p^e)a_i.
\end{equation}

Write $G=\sum_ib_iD_i$. Set
\[\delta_{i,\varepsilon}=\begin{cases} \varepsilon &{\rm when}\ D_i{\rm  is\ a\ component\ of}E\\ 0 &{\rm otherwise}. \end{cases}\]
Since $K_{X'}=-\sum_iD_i$, we can write
\[\lceil K_{X'}-\pi^*(K_X+\Delta)- tG +\varepsilon E\rceil=\sum_i\lceil -1-a_i-tb_i+\delta_{i,\varepsilon}\rceil D_i .\]
Note that
\begin{equation}
\label{coefficients-JE}
a_i+tb_i+\lceil -1-a_i-tb_i+\delta_{i,\varepsilon}\rceil \geq -1
\end{equation}

Let $x^v$ be an arbitrary element of $\mathcal{J}(E)$. We wish to prove that $\phi^n(x^u\cdot x^v)$ is in $\mathcal{J}(E)$.  It is clear that \[\phi^n(x^u\cdot x^v)=\begin{cases} x^{\frac{u+v-\frac{p^{en}-1}{p^e-1}w}{p^{en}}} & {\rm if}\ \frac{u+v-\frac{p^{en}-1}{p^e-1}w}{p^{en}}\in M \\ 0 & {\rm otherwise}.\end{cases}\]
If $\phi^n(x^u\cdot x^v)=0$, then clearly $\phi^n(x^u\cdot x^v)\in \mathcal{J}(E)$. Assume that $\widetilde{u}:=\frac{u+v-\frac{p^{en}-1}{p^e-1}w}{p^{en}}$ is in $M$. Then we must have $\langle \tilde{u}, v_i\rangle \in \mathbb{Z}$ for each $v_i$. Now we have
\begin{align}
 & \langle \widetilde{u}, v_i \rangle + \lceil -1-a_i-tb_i+\delta_{i,\varepsilon}\rceil \notag\\
 &=\frac{1}{p^{en}}\langle u,v_i\rangle + \frac{1}{p^{en}}\langle v,v_i\rangle -\frac{1}{p^{en}}\left\langle \frac{p^{en}-1}{p^e-1}w,v_i\right\rangle + \lceil -1-a_i-tb_i+\delta_{i,\varepsilon}\rceil \notag\\
 &\geq \frac{1}{p^{en}}t(p^{en}-1)b_i - \frac{1}{p^{en}} \lceil -1-a_i-tb_i+\delta_{i,\varepsilon}\rceil + \frac{1}{p^{en}}(p^{en}-1)a_i + \lceil -1-a_i-tb_i+\delta_{i,\varepsilon}\rceil \notag\\
& = \frac{p^{en}-1}{p^{en}}(tb_i+a_i+\lceil -1-a_i-tb_i+\delta_{i,\varepsilon}\rceil)\notag\\
& \geq  \frac{p^{en}-1}{p^{en}}\cdot(-1),\ {\rm by\,} (\ref{coefficients-JE})\notag\\
&> -1\notag
\end{align}
But we already saw that $\langle \widetilde{u}, v_i\rangle \in \mathbb{Z}$ and hence $ \langle \widetilde{u}, v_i \rangle + \lceil -1-a_i-tb_i+\delta_{i,\varepsilon}\rceil $ must be an integer, therefore we must have
\[ \langle \widetilde{u}, v_i \rangle + \lceil -1-a_i-tb_i+\delta_{i,\varepsilon}\rceil \geq 0\]
i.e.
\[\Div_{X'}(\phi^n(x^u\cdot x^v))+\lceil K_{X'}-\pi^*(K_X+\Delta)- tG +\varepsilon E\rceil \geq 0.\]
So we have $\phi^n(x^u\cdot x^v)\in \mathcal{J}(E)$ as desired. This completes the proof that
\[
\sum_{n>0} \phi^n(\overline{\ba^{t(p^{en}-1)}}\mathcal{J}(E))\subseteq \mathcal{J}(E).
\]

Next we prove that, given any $x^v\in \mathcal{J}(E)$, there is an integer $n$ and an element $z\in \overline{\ba^{t(p^{en}-1)}}\cdot \mathcal{J}(E)$ such that $\phi^n_{\Delta}(z)=x^v$. Set \[z=x^{(p^{en}-1)v+\frac{p^{en}-1}{p^e-1}w}\cdot x^v.\]
It is clear that $\phi^n(z)=x^v$ and $x^v$ is already in $\mathcal{J}(E)$. It remains to show that $x^{(p^{en}-1)v+\frac{p^{en}-1}{p^e-1}w}$ is in $\overline{\ba^{t(p^{en}-1)}}=\pi_* \O_{X'}(-t(p^{en}-1)G)$.

It is clear we have that
\begin{align}
& \Div_{X'}(x^{(p^{en}-1)v+\frac{p^{en}-1}{p^e-1}w})\notag\\
&=\Div_{X'}(x^{(p^{en}-1)v})+\Div_{X'}(x^{\frac{p^{en}-1}{p^e-1}w})\notag\\
&=\Div_{X'}(x^{(p^{en}-1)v})+\sum_i \langle \frac{p^{en}-1}{p^e-1}w,v_i\rangle D_i \notag\\
&=\Div_{X'}(x^{(p^{en}-1)v})  - (p^{en}-1)\sum_ia_iD_i,{\rm because\ of\ }(\ref{coefficients-w})\notag\\
&\geq - (p^{en}-1)\lceil \left(\sum_i(-1-a_i-tb_i)D_i\right)+\varepsilon E\rceil -(p^{en}-1)\sum_ia_iD_i,{\rm\, since\ }x^v\in \mathcal{J}(E)\notag\\
&= - (p^{en}-1)\lceil \sum_i(-1-a_i-tb_i+\delta_{i,\varepsilon})D_i\rceil -(p^{en}-1)\sum_ia_iD_i\notag\\
&=(p^{en}-1)\sum_i(-\lceil -1-a_i-tb_i+\delta_{i,\varepsilon}\rceil -a_i)D_i.\notag
\end{align}

Since it is easy to check that $-\lceil -1-a_i-tb_i+\delta_{i,\varepsilon}\rceil -a_i\geq tb_i$ for $0<\varepsilon \ll1$, we have $\Div_{X'}(x^{(p^{en}-1)v+\frac{p^{en}-1}{p^e-1}w})\in \O_{X'}(-t(p^{en}-1)G)$ as desired. This finishes the proof of our theorem.
\end{proof}

\begin{corollary}
Suppose that $X$ is an affine toric variety in characteristic zero or characteristic $p > 0$, that $\ba$ is a toric ideal of $\O_X$ and $\Delta$ is an effective toric $\bQ$-divisor such that $K_X + \Delta$ is $\bQ$-Cartier so that $l(K_X + \Delta) = \Div(x^m)$ for some integer $l > 0$.  Set $\sF$ to be the set of all faces of $t\Newt(\ba)$ and for any $\tau \in \sF$, set \[ K_{\tau} = \langle x^v | v \in \relint( {m \over l} + \tau )\cap S\rangle .\]  Then a non-zero ideal $I \subseteq k[S]$ is in $\sJ_Z(X, \Delta,\ba^t)$ if and only if there exists some subset $\sG \subseteq \sF$ such that
\[I = \sum_{\tau \in \sG} \left( \sum_{\tau \subseteq \tau' \text{ in } \sF} K_{\tau'} \right)\]
\end{corollary}
\begin{proof}
The statement in characteristic zero reduces to the characteristic $p > 0$ statement from reduction to characteristic $p \gg 0$.
For the characteristic $p > 0$ statement, suppose that $l(K_X + \Delta) = \Div(x^m)$ in characteristic $p \gg 0$.
Then if $\phi_{\Delta}(\blank) = \phi_c(x^{-w} \cdot \blank)$, we already saw that $\Delta = (-K_X) + {1 \over p^e - 1} \Div_X(x^{-w})$, thus $\Div(x^{-w}) = (p^e - 1)(K_X + \Delta)$ so that ${w \over 1 - p^e} = {m \over l}$ and the result follows
by Theorems \ref{cartier-algebra-fixed-ideal}, \ref{fixed-ideals-are-adjoint-ideals}, and \ref{adjoint-ideals-are-fixed}.
\end{proof}

\begin{corollary}
Suppose that $X$ is an affine toric variety in characteristic zero or characteristic $p > 0$ and that $\Delta$ is an effective toric $\bQ$-divisor such that $K_X + \Delta$ is $\bQ$-Cartier.  Then the elements of $\sJ_Z(X, \Delta,\ba^t)$ are closed under summation.
\end{corollary}

\bibliographystyle{skalpha}
\bibliography{CommonBib}

\end{document}